\documentclass[12pt]{amsart}
\usepackage{hyperref} 
\usepackage{amsmath, amsthm, amssymb}
\usepackage{hyperref} 
\usepackage{enumerate}
\usepackage{verbatim}
\usepackage{esint}
\usepackage[T1]{fontenc}
\usepackage{cleveref}

\usepackage{tikz,amsthm,amsmath,amstext,amssymb,amscd,epsfig,euscript, mathrsfs,
 dsfont,pspicture,
multicol,graphpap,graphics,graphicx,times,enumerate,subfig,
sidecap,
wrapfig,color,pict2e}
\usepackage{setspace}
\usepackage[isbn=false, url=false]{biblatex} 
\usepackage{todonotes}

\numberwithin{equation}{section}

\title{Factorization and piecewise affine approximation of bi-Lipschitz mappings on large sets}
\date{\today}
\author{Guy C. David}
\address{Department of Mathematical Sciences\\ Ball State University, Muncie, IN 47306}
\email{gcdavid@bsu.edu}

\author{Matthew Romney}
\address{Department of Mathematical Sciences\\ Stevens Institute of Technology, Hoboken, NJ 07030}
\email{mromney@stevens.edu}

\author{Raanan Schul}
\address{Department of Mathematics\\ Stony Brook University\\ Stony Brook, NY 11794}
\email{raanan.schul@stonybrook.edu}

\thanks{G. C. David was partially supported by the National Science Foundation under Grant No. DMS-2054004. M. Romney was partially supported by the National Science Foundation under Grant No. DMS-2413156. R. Schul was partially supported by the National Science Foundation under Grant No. DMS-2154613.}


\begin{document}
\maketitle

\theoremstyle{plain}
\newtheorem{theorem}{Theorem}
\newtheorem{exercise}{Exercise}
\newtheorem{corollary}[theorem]{Corollary}
\newtheorem{scholium}[theorem]{Scholium}
\newtheorem{claim}[theorem]{Claim}
\newtheorem{lemma}[theorem]{Lemma}
\newtheorem{sublemma}[theorem]{Lemma}
\newtheorem{proposition}[theorem]{Proposition}
\newtheorem{conjecture}[theorem]{Conjecture}
\newtheorem{maintheorem}{Theorem}
\newtheorem{maincor}[maintheorem]{Corollary}
\newtheorem{mainproposition}[maintheorem]{Proposition}
\renewcommand{\themaintheorem}{\Alph{maintheorem}}

\theoremstyle{definition}
\newtheorem{fact}[theorem]{Fact}
\newtheorem{example}[theorem]{Example}
\newtheorem{definition}[theorem]{Definition}
\newtheorem{remark}[theorem]{Remark}
\newtheorem{question}[theorem]{Question}

\numberwithin{equation}{section}
\numberwithin{theorem}{section}

\newcommand{\allU}{\mathcal{U}}
\newcommand{\badU}{\mathcal{V}}
\newcommand{\goodU}{\mathcal{W}}

\newcommand{\cG}{\mathcal{G}}
\newcommand{\RR}{\mathbb{R}}
\newcommand{\HH}{\mathcal{H}}
\newcommand{\LIP}{\textnormal{LIP}}
\newcommand{\Lip}{\textnormal{Lip}}
\newcommand{\Tan}{\textnormal{Tan}}
\newcommand{\length}{\textnormal{length}}
\newcommand{\dist}{\textnormal{dist}}
\newcommand{\diam}{\textnormal{diam}}
\newcommand{\vol}{\textnormal{vol}}
\newcommand{\rad}{\textnormal{rad}}
\newcommand{\side}{\textnormal{side}}
\def\graph{{\rm graph}}
	\def\proj{{\rm proj}}

\newcommand{\GL}{\text{GL}}
\newcommand{\SO}{\text{SO}}
\def\bA{{\mathbb{A}}}
\def\bB{{\mathbb{B}}}
\def\bC{{\mathbb{C}}}
\def\bD{{\mathbb{D}}}
\def\bR{{\mathbb{R}}}
\def\bS{{\mathbb{S}}}
\def\bO{{\mathbb{O}}}
\def\bE{{\mathbb{E}}}
\def\bF{{\mathbb{F}}}
\def\bH{{\mathbb{H}}}
\def\bI{{\mathbb{I}}}
\def\bT{{\mathbb{T}}}
\def\bZ{{\mathbb{Z}}}
\def\bX{{\mathbb{X}}}
\def\bP{{\mathbb{P}}}
\def\bN{{\mathbb{N}}}
\def\bQ{{\mathbb{Q}}}
\def\bK{{\mathbb{K}}}
\def\bG{{\mathbb{G}}}

\def\nrj{{\mathcal{E}}}
\def\cA{{\mathscr{A}}}
\def\cB{{\mathscr{B}}}
\def\cC{{\mathscr{C}}}
\def\cD{{\mathscr{D}}}
\def\cE{{\mathscr{E}}}
\def\cF{{\mathscr{F}}}
\def\cG{{\mathscr{G}}}
\def\cH{{\mathcal{H}}}
\def\cI{{\mathscr{I}}}
\def\cJ{{\mathscr{J}}}
\def\cK{{\mathscr{K}}}
\def\Layer{{\rm Layer}}
\def\cM{{\mathscr{M}}}
\def\cN{{\mathscr{N}}}
\def\cO{{\mathscr{O}}}
\def\cP{{\mathscr{P}}}
\def\cQ{{\mathcal{Q}}}
\def\cR{{\mathscr{R}}}
\def\cS{{\mathscr{S}}}
\def\Up{{\rm Up}}
\def\cU{{\mathscr{U}}}
\def\cV{{\mathscr{V}}}
\def\cW{{\mathscr{W}}}
\def\cX{{\mathscr{X}}}
\def\cY{{\mathscr{Y}}}
\def\cZ{{\mathscr{Z}}}

  \def\del{\partial}
  \def\diam{{\rm diam}}
	\def\VV{{\mathcal{V}}}
	\def\FF{{\mathcal{F}}}
	\def\QQ{{\mathcal{Q}}}
	\def\BB{{\mathcal{B}}}
	\def\XX{{\mathcal{X}}}
	\def\PP{{\mathcal{P}}}

  \def\del{\partial}
  \def\diam{{\rm diam}}
	\def\image{{\rm Image}}
	\def\domain{{\rm Domain}}
  \def\dist{{\rm dist}}
	\newcommand{\Gr}{\mathbf{Gr}}
\newcommand{\md}{\textnormal{md}}
\newcommand{\vspan}{\textnormal{span}}

\begin{abstract}
    A well-known open problem asks whether every bi-Lipschitz homeomorphism of $\mathbb{R}^d$ factors as a composition of mappings of small distortion. We show that every bi-Lipschitz embedding of the unit cube $[0,1]^d$ into $\mathbb{R}^d$ factors into finitely many global bi-Lipschitz mappings of small distortion, outside of an exceptional set of arbitrarily small Lebesgue measure, which cannot in general be removed. Our main tool is a corona-type decomposition theorem for bi-Lipschitz mappings.
    As corollaries, we obtain a related factorization result for bi-Lipschitz homeomorphisms of the $d$-sphere, and we show that bi-Lipschitz embeddings of the unit $d$-cube in $\mathbb{R}^d$ can be approximated by global piecewise affine homeomorphisms outside of a small set.
\end{abstract}

\tableofcontents

\section{Introduction}

This paper is motivated by the \emph{factorization problem} for bi-Lipshitz mappings: can any bi-Lipschitz homeomorphism of $\mathbb{R}^d$ or $\mathbb{S}^d$ ($d \geq 2$) be written as a composition of finitely many bi-Lipschitz mappings with distortion close to $1$? In full generality, this is a well-known open problem; see, e.g., Conjecture 1.1 of \cite{FletcherMarkovic} or \cite[p. 184]{AstalaIwaniecMartin}. 

An affirmative answer assuming that the map is a $C^1$-diffeomorphism has been given by Fletcher--Markovic \cite{FletcherMarkovic}. Note, however, that a bi-Lipschitz map need not be everywhere differentiable. A basic example is provided by the logarithmic spiral map in the plane, given in polar coordinates by $(r,\theta) \mapsto (r,\theta + k \log r)$ for some parameter $k>0$. The logarithmic spiral indeed factors into bi-Lipschitz maps of small distortion, as shown by Freedman--He \cite{FreedmanHe}, although with more factors than one might naively expect; see \cite{GutlyanskiiMartio} and \cite[Theorem 6.4]{AIPS} for a sharp result on the number of factors. More recently, Fletcher--Vellis have shown how certain ``multi-twist'' maps can be factored into maps of distortion close to $1$ \cite{FletcherVellis}.

In this paper, we consider a weaker version of the factorization problem in which we are allowed to ignore a ``exceptional set'' of arbitrarily small Lebesgue measure. We also consider bi-Lipschitz embeddings of the unit cube $[0,1]^d$ into $\RR^d$ rather than globally defined bi-Lipschitz homeomorphisms (although we discuss the global case further in subsection \ref{subsec:sphere}). Our main result is the following.

\begin{theorem}\label{thm:factorization}
Let $f\colon [0,1]^d \rightarrow \RR^d$ be a bi-Lipschitz embedding, and $\delta, \epsilon>0$. Then there are $(1+\epsilon)$-bi-Lipschitz homeomorphisms $f_1, \dots, f_T$ of $\RR^d$ and a set $E\subseteq [0,1]^d$ such that
$$ f = f_T \circ \dots \circ f_1 \text{ on } [0,1]^d\setminus E$$
and
$$ |E|< \delta.$$
The number $T$ of mappings required depends only on $\delta$, $\epsilon$, $d$, and the distortion of $f$.
\end{theorem}

Here, $|E|$ denotes the Lebesgue measure of $E$. Note that, although the map $f$ is factored only on a certain large subset of its domain, the factoring mappings $f_i$ are globally defined.

An important feature of Theorem \ref{thm:factorization} is that it is \textit{quantitative}, with the number of mappings in the factorization bounded in a manner depending only on the given parameters and the bi-Lipschitz distortion of $f$, and otherwise not on the map $f$ itself.

\begin{remark}\label{rmk:garbage}
The ``exceptional set'' $E$ cannot in general be removed entirely from Theorem \ref{thm:factorization}. If it could, then the theorem would imply that every bi-Lipschitz embedding of the $d$-dimensional unit cube (or, equivalently, the unit ball) into $\RR^d$ extends to a global homeomorphism of $\RR^d$. (This is because the mappings $f_1, \dots, f_T$ are global homeomorphisms of $\RR^d$.) However, as was first observed by Gehring and proven by Martin in \cite[Theorem 3.7]{Martin}, this is not the case: a version of the Fox--Artin construction yields a bi-Lipschitz embedding of the unit ball of $\RR^3$ into $\RR^3$ that does not extend to a homeomorphism of $\RR^3$.

On the other hand, as far as we know, it is possible that $E$ can be removed entirely when $d=2$. For the general case, it is also conceivable that $E$ can be required to lie in a small neighborhood of the boundary of $[0,1]^d$, but this is not what our proof of Theorem \ref{thm:factorization} provides.
\end{remark}

We now describe the main ideas that go into proving Theorem \ref{thm:factorization}. Similarly to other quantitative results on bi-Lipschitz and related classes of mappings, the major tool used to prove Theorem \ref{thm:factorization} is the David--Semmes theory of uniform rectifiability. A basic principle of this theory is that uniformly rectifiable sets admit ``corona decompositions'': certain multi-scale descriptions in terms of pieces that are well-approximated by Lipschitz graphs and which were introduced by Semmes \cite{semmes1990analysis} and David-Semmes \cite{david1991dela}.

The definition of corona decomposition given in \cite[Definition I.3.19]{DavidSemmes} is stated in terms of \emph{sets}; for our purposes, it is important to have an analogue of this result for bi-Lipschitz maps. This is given as Proposition \ref{prop:coronafunction} below, and it may be of independent interest. Roughly speaking, we show that a bi-Lipschitz embedding of $[0,1]^d$ into $\RR^d$ is well-approximated on many scales and locations by a collection of so-called \emph{almost affine} mappings. Here an ``almost affine'' mapping is the composition of an affine mapping with a $(1+\eta)$-bi-Lipschitz mapping, where $\eta>0$ is small; see Definition \ref{def:almostaffine}. In particular, there must be many subsets of the domain (in a precise sense) where the restriction of a bi-Lipschitz mapping is almost affine.

Since affine mappings are not difficult to factor into mappings of small distortion, this gives an entry to the proof of Theorem \ref{thm:factorization}. The other component of the proof is a detailed scheme that collects the various factorizations of these almost affine maps into a global factorization of the original map off of a small set. This is the content of Section \ref{sec:mainproof}.

\subsection{Piecewise affine approximation}
Theorem \ref{thm:factorization} and the general problem of factoring bi-Lipschitz mappings are closely related to another topic in geometric topology: the approximation of homeomorphisms by piecewise affine homeomorphisms. Recall that a mapping $g\colon\RR^d\rightarrow\RR^d$ is \textit{piecewise affine} if there is a locally finite triangulation of $\RR^d$ such that $g$ is affine on each simplex of the triangulation. 

The basic problem here is whether homeomorphisms between domains in $\RR^d$ can be approximated uniformly by piecewise affine homeomorphisms. Because we also have an eye towards quantitative issues, we may also consider how the ``complexity'' of the piecewise affine approximation depends on the error of approximation. For a piecewise affine map $g\colon \RR^d\rightarrow\RR^d$, a simple way to measure this complexity is to set  $N(g, [0,1]^d)$ to be the minimum integer $N$ such that there is a triangulation of $[0,1]^d$ into $N$ simplices on which the restriction of $g$ is affine.

The connection with Theorem \ref{thm:factorization} is due to the following fact: If $g\colon \RR^d\rightarrow\RR^d$ is a $(1+\epsilon)$-bi-Lipschitz mapping for sufficiently small $\epsilon>0$, then $g$ can be approximated arbitrarily well in the uniform norm by piecewise affine homeomorphisms. This fact may be known, but we provide a complete proof in Proposition \ref{prop:PL} below.

As a consequence, we obtain the following corollary, after whose statement we provide some broader context.

\begin{corollary}\label{cor:PLapproximation}
Let $f$ be an $L$-bi-Lipschitz embedding of $[0,1]^d$ into $\RR^d$ and let $\delta>0$. Then there is a set $E\subseteq [0,1]^d$ with  $|E|<\delta$ such that $f$ can be arbitrarily well-approximated in the supremum norm by piecewise affine homeomorphisms off of $E$.

More precisely, for each $\eta>0$, there is a bi-Lipschitz, piecewise affine homeomorphism $g\colon \RR^d\rightarrow \RR^d$ such that
$$ \sup_{x\in [0,1]^d\setminus E}|g(x)-f(x)| \leq \eta.$$
The bi-Lipschitz constant of $g$ can be bounded by a constant depending only on $d$, $\delta$, and $L$. 

Moreover, we may bound $N(g,[0,1]^d)$ by a constant depending only on $\eta$, $d$, $\delta$, and $L$.
\end{corollary}

\begin{remark}\label{rmk:PLapproximation}
In all dimensions $d\neq 4$, it is known that a topological embedding of the open unit cube of $\RR^d$ into $\RR^d$ (and therefore a bi-Lipschitz embedding of the closed unit cube) can be approximated in the supremum norm by piecewise affine homeomorphisms.  See the discussion in \cite[p. 1399]{MoraPratelli}. Thus, the qualitative conclusion of Corollary \ref{cor:PLapproximation} is weaker than known results if $d\neq 4$. However, even in this case, the quantitative bounds on the ``complexity'' $N(g,[0,1]^d)$ of the approximating piecewise affine maps are new, as far as we know. 

In dimension $4$, it follows from a result of Donaldson and Sullivan \cite[p. 183]{DonaldsonSullivan} that there is a topological embedding of the open unit cube of $\RR^4$ into $\RR^4$ that cannot be approximated by piecewise affine homeomorphisms.

When $d=4$, it seems to be an open question whether every bi-Lipschitz embedding of $[0,1]^4$ in $\RR^4$ can be approximated arbitrarily well in the supremum norm by piecewise affine homeomorphisms on all of $[0,1]^4$. Corollary \ref{cor:PLapproximation} gives some partial information in this direction.
\end{remark}

\subsection{Homeomorphisms of the sphere}\label{subsec:sphere}
It is possible to use Theorem \ref{thm:factorization} to obtain a factorization result for bi-Lipschitz homeomorphisms of the sphere, which is more in line with \cite{FletcherMarkovic}.

We follow \cite{FletcherMarkovic} and consider the chordal metric $\chi$ on a sphere of unit diameter $\mathbb{S}^d\subseteq \RR^{d+1}$. As a corollary of Theorem \ref{thm:factorization} (and Remark \ref{rmk:factorization}), we have the following:

\begin{corollary}\label{cor:sphere}
Let $f\colon \mathbb{S}^d \rightarrow \mathbb{S}^d$ be a bi-Lipschitz homeomorphism (with respect to $\chi$), and fix $\delta, \epsilon>0$. Then there are $(1+\epsilon)$-bi-Lipschitz homeomorphisms $f_1, \dots, f_T$ of $\mathbb{S}^d$ and a set $E\subseteq \mathbb{S}^d$ such that
$$ f = f_T \circ \dots \circ f_1 \text{ on } \mathbb{S}^d\setminus E$$
and
$$ \mathcal{H}^d(E) < \delta.$$
The number $T$ of mappings required depends only on $\delta$, $\epsilon$, $d$, and the distortion of $f$. (Here $\mathcal{H}^d(E)$ is the $d$-dimensional Hausdorff measure of the set $E$ in the metric $\chi$.)
\end{corollary}

\begin{remark}
We observed in Remark \ref{rmk:garbage} that the ``exceptional set'' $E$ cannot be completely removed in Theorem \ref{thm:factorization}. By contrast, it is an open problem whether it can be removed in Corollary \ref{cor:sphere}. This is exactly the setting discussed by Fletcher and Markovic in \cite{FletcherMarkovic}, who show (in a non-quantitative sense) that it can be removed for $C^1$-diffeomorphisms.
\end{remark}

\begin{remark}
A weaker and \textit{non-quantitative} variant of Corollary \ref{cor:sphere} can be deduced from the existing literature as follows. Suppose $f\colon \mathbb{B}\rightarrow\mathbb{B}$ is a bi-Lipschitz homeomorphism of the $d$-dimensional unit ball in the chordal metric $\chi$ (which we view as the lower hemisphere of $\mathbb{S}^d$).  As a consequence of a theorem of White \cite[Theorem 2]{White}, there is a $C^1$-diffeomorphism of $\mathbb{B}$ that agrees with $f$ on a large subset of the ball, and hence a $C^1$-diffeomorphism of the sphere agreeing with $f$ on a large subset of $\mathbb{B}$. This diffeomorphism of the sphere can then be factored by the above-mentioned result of Fletcher-Markovic. This gives a factorization of the original homeomorphism of $\mathbb{B}$ off of a set of small measure. This argument seems to require a homeomorphism that fixes an appropriate domain of the sphere, and does not yield quantitative dependence on the parameters.

\end{remark}

\subsection{Outline of the paper}
In Section \ref{sec:prelim}, we give some basic definitions and facts, including properties of the local topological degree needed for later proofs. Section \ref{sec:almostaffine} contains the definition of almost affine mappings and Proposition \ref{prop:coronafunction}, which is a type of corona decomposition for bi-Lipschitz embeddings. Section \ref{sec:extension} gives some results on extensions and gluings of mappings, and Section \ref{sec:factorlinear} shows how to factor affine mappings ``locally''. The proofs of Theorem \ref{thm:factorization},  Corollary \ref{cor:PLapproximation}, and Corollary \ref{cor:sphere} are then given in Sections \ref{sec:mainproof}, \ref{sec:PL}, and \ref{sec:sphere}, respectively.

\subsection*{Acknowledgments}
The authors would like to thank Dennis Sullivan for helpful conversations, especially regarding the literature on approximation by piecewise affine homeomorphisms discussed in Remark \ref{rmk:PLapproximation}. 

\section{Preliminaries}\label{sec:prelim}
\subsection{Basic notation}
Let $A\subseteq \RR^d$. A mapping $f\colon A \rightarrow \RR^n$ is called a \textit{bi-Lipschitz embedding}, or simply \textit{bi-Lipschitz}, if there is a constant $L\geq 1$ such that
$$ L^{-1}|x-y|\leq |f(x)-f(y)|\leq L|x-y| \text{ for all } x,y\in A.$$
If $d=n$ and $A=\mathbb{R}^d$, then $f$ is necessarily surjective and we say that $f$ is a \textit{bi-Lipschitz homeomorphism of $\mathbb{R}^d$}.

If we wish to emphasize the constant, we say that $f$ is \textit{$L$-bi-Lipschitz}. The minimal possible $L$ for which $f$ is $L$-bi-Lipschitz is called the \textit{bi-Lipschitz constant} or \textit{distortion} of $f$.

Let $Q \subset \mathbb{R}^d$ be a cube. We let $x(Q)$ denote the center of $Q$ and $\ell(Q)$ denote the side length of $Q$. We also use the notation $\mathcal{C}(x,\ell)$ to denote the cube with center $x$ and side length $\ell$. Given $\lambda>0$, $\lambda Q$ denotes the cube with the same center as $Q$ with faces parallel to $Q$ and side length $\lambda \ell(Q)$. More generally, if $A \colon \mathbb{R}^d \to \mathbb{R}^d$ is an affine map and $R = A(Q)$, then $\lambda R$ denotes the set $A(\lambda Q)$. We say that two cubes \textit{overlap} if their interiors intersect. We also assume throughout the paper that all cubes have faces parallel to the coordinate planes.

\subsection{Local degree}\label{subsec:localdegree}
Suppose that $F\colon \RR^d\rightarrow \RR^d$ is a continuous mapping, $D\subseteq \RR^d$ is a bounded domain, and $y\in \RR^d\setminus F(\partial D).$ This triple $(y,D,F)$ induces an integer $\mu(y,D,f)$ called the \textit{local degree} of $F$ at $y$ with respect to $D$. We refer to \cite[Section 2]{HeinonenRickman} for the definition of the local degree. The following properties of local degree are either immediate from the definition in \cite{HeinonenRickman} or are listed there.
\begin{enumerate}[(i)]
\item \label{deg:local} If $F|_{\overline{D}} = G|_{\overline{D}}$, then
$$ \mu(y, D, F) = \mu(y,D, G)$$
for all $y\in\RR^d\setminus F(\partial D)$.
\item\label{deg:homeo} If $F|_D \colon D \rightarrow F(D)$ is a homeomorphism and $y\in F(D)$, then $\mu(y,D,F)=\pm 1$.
\item\label{deg:constant} The integer $\mu(y,D,F)$ remains constant as $y$ varies within a connected component of $\RR^d\setminus F(\partial D)$.
\item\label{deg:homotopic}  If $F,G\colon \RR^d \rightarrow \RR^d$ are homotopic through proper maps $H_t:\RR^d\rightarrow \RR^d$ ($0\leq t \leq 1$) and $y\notin H_t(\partial D)$ for all $t$, then
$$ \mu(y,D,F) = \mu(y,D,G).$$
\item\label{deg:notin} If $y\notin F(\overline{D})$, then $\mu(y,D,F)=0$.
\item\label{deg:preimage} If $y\in\RR^d \setminus F(\partial D)$ and if $F^{-1}(y) \subseteq D_1 \cup \dots \cup D_m$, where $D_i$ are disjoint domains in $D$ with $y\notin F(\partial D_i)$, then 
$$ \mu(y,D,F) = \sum_{i=1}^m \mu(y, D_i, F).$$
\end{enumerate}
A mapping $F\colon\RR^d\rightarrow\RR^d$ is called \textit{orientation-preserving} if 
$$ \mu(y,D, F) > 0$$
whenever $D\subset \RR^d$ is a relatively compact domain and $y\in F(D) \setminus F(\partial D)$; it is called \textit{orientation-reversing} if this inequality is always reversed. A homeomorphism of $\RR^d$ must be either orientation-preserving or orientation-reversing.

The following basic lemma, a consequence of property \eqref{deg:homotopic} of local degree, will be used repeatedly in Section \ref{sec:PL}.

\begin{lemma}\label{lem:closedegree}
Suppose that $D$ is a bounded domain in $\RR^d$ and $h_1$ and $h_2$ are two proper continuous mappings from $\RR^d$ to $\RR^d$. Suppose that $p\in \RR^d$ satisfies
$$ \sup_{x\in \partial D} |h_1(x) - h_2(x)| < \dist(p,h_1(\partial D) \cup h_2(\partial D)).$$
Then
$$ \mu(p, D, h_1) = \mu(p, D, h_2).$$
\end{lemma}
\begin{proof}
By the locality property \eqref{deg:local}, we may  modify $h_2$ (without changing $\mu(p, D, h_2)$) so that it remains continuous and agrees with $h_1$ outside of a large ball containing $D$. This is to ensure that the obvious homotopy between $h_1$ and $h_2$ remains proper at each step. We then apply that ``straight line'' homotopy 
$$ H_t(x) = (1-t)h_1(x)+th_2(x)$$
and the homotopy invariance property \eqref{deg:homotopic} of local degree. The assumption ensures that
$$ p\notin H_t(\partial D)$$
for all $t$.
\end{proof}

\section{Almost affine maps and coronizations}\label{sec:almostaffine}

\subsection{Almost affine maps}

The goal of this section is Proposition \ref{prop:coronafunction}, which states that bi-Lipschitz mappings can be well-approximated on most locations and scales, in a precise sense, by a certain class of mappings that we call \emph{almost affine}.

\begin{definition}\label{def:almostaffine}
Let $Q$ be a cube in $\mathbb{R}^d$. We say that a map $g\colon Q \rightarrow \RR^d$ is \textit{$\eta$-almost affine} if there is an affine map $A\colon\RR^d\rightarrow\RR^d$ and a $(1+\eta)$-bi-Lipschitz map $\phi\colon \RR^d\rightarrow \RR^d$ such that
\[g = \phi \circ A \text{ on } Q\]

If $A$ is $L$-bi-Lipschitz then we say that $g$ is $(L,\eta)$-almost affine.
\end{definition}

We next observe that two additional properties \eqref{eq:almostaffine} and \eqref{eq:almostaffine2} of almost affine maps can be arranged, as a consequence of the other properties by modifying the maps $A$ and $\phi$.
\begin{lemma}\label{lem:almostaffine}
Let $Q$ be a cube in $\mathbb{R}^d$, $y\in\mathbb{R}^d$, and $g\colon Q\rightarrow \RR^d$. Suppose that there is an $L$-bi-Lipschitz affine map $A$ and a $(1+\eta)$-bi-Lipschitz map $\phi\colon\RR^d\rightarrow\RR^d$ such that
$$ g = \phi \circ A \text{ on } Q.$$
Then there is an $L$-bi-Lipschitz affine map $\tilde{A}$ and a $(1+\eta)$-bi-Lipschitz map $\tilde{\phi}\colon\RR^d\rightarrow\RR^d$ such that
$$ g = \tilde{\phi}\circ \tilde{A} \text{ on } Q,$$
\begin{equation}\label{eq:almostaffine}
 |\tilde{\phi}(x)-x|\lesssim_d L\eta \diam(Q) + |g(x(Q))-y| \text{ for all }x\in \tilde{A}(Q),
\end{equation}
and
\begin{equation}\label{eq:almostaffine2}
    \tilde{A}(x(Q))=y.
\end{equation}
\end{lemma}
In other words, by modifying the factorization of an almost affine map $g$ (but not $g$ itself) we can arrange that the almost-isometric factor is additively close to the identity and that the affine piece sends the center of $Q$ to a prescribed point.

Lemma \ref{lem:almostaffine} is a consequence of a result of John \cite{John:61}; we use the following formulation of David--Toro \cite{DT:99}:
\begin{lemma}[\cite{DT:99}, Lemma 7.11]\label{lem:DavidToro}
If $\tilde{B}$ is a subset of $\RR^d$ such that
$$ \overline{B}(0,1) \subset \tilde{B} \subset B(0,10) $$
and if $g\colon \tilde{B}\rightarrow \RR^d$ satisfies
$$ ||g(x)-g(y)| - |x-y||\leq \alpha$$
for all $x,y\in\tilde{B}$ and some $\alpha>0$ sufficiently small, then there is an isometry $J$ of $\RR^n$ such that 
$$ |J(x)-g(x)|\leq C\alpha \text{ for all } x\in\tilde{B}.$$
The constant $C$ depends only on $d$.
\end{lemma}

In particular, the lemma holds if $\tilde{B}=\overline{B}(0,1).$

\begin{proof}[Proof of Lemma \ref{lem:almostaffine}]
We may assume without loss of generality that $0$ is the center of $Q$, and that $g(0)=A(0)=\phi(0)=0$. We then have $A(Q) \subset B(0,L\diam(Q))$. Also assume first that $y=g(x(Q))=0$.

Consider the map $h(x)= (L\diam(Q))^{-1}\phi(L\diam(Q) x)$ on $\tilde{B}=\overline{B}(0,1).$ It satisfies
$$ \left||h(x)-h(y)| - |x-y|\right|\leq 2\eta$$
for all $x,y\in\tilde{B}$. Let $J$ be the associated isometry from Lemma \ref{lem:DavidToro}, so that
$$ |J(x)-h(x)| \leq 2C\eta \text{ for all } x\in\tilde{B}.$$
Since $h(0)=0$, we can also take $J$ to have $J(0)=0$, at the cost of changing $2C\eta$ to $4C\eta$ in the above inequality. This makes $J$ linear.

Set $\tilde{\phi} = \phi \circ J^{-1}$ and $\tilde{A}= J\circ A$. We still have $g = \tilde{\phi}\circ \tilde{A}$ on $Q$, and, since $J$ is an isometry, the map $\tilde{\phi}$ is $(1+\eta)$-bi-Lipschitz.
. 

Let $x\in \tilde{A}(Q)$ and $y=J^{-1}(x)\in A(Q)\subseteq B(0,L\diam(Q))$. Thus
$$ (L\diam(Q))^{-1} y \in \tilde{B}.$$
Therefore
\begin{align*}
|\tilde{\phi}(x)-x| &= |\phi(y) - J(y)|\\
&= |(L\diam(Q))h((L\diam(Q))^{-1}y) - J(y)|\\
&= (L\diam(Q))|h((L\diam(Q))^{-1}y) - J((L\diam(Q))^{-1}y)|\\
&\leq CL\diam(Q)\eta.
\end{align*}
This verifies \eqref{eq:almostaffine} and \eqref{eq:almostaffine2} under the assumption that $y=g(x(Q))$.

If $y\neq g(x(Q))$, we apply the construction above, and then modify it by replacing $\tilde{A}$ with $\tilde{A} - \tilde{A}(x(Q)) + y$ and $\tilde{\phi}$ with $\tilde{\phi} + \tilde{A}(x(Q))-y$. This fixes \eqref{eq:almostaffine2}. To see \eqref{eq:almostaffine}, note that the distance from $\tilde{\phi}$ to the identity has increased by at most
$$ |A(x(Q)) - y| \lesssim L\eta\diam(Q) + |g(x(Q))-y|. $$
\end{proof}

\subsection{Corona decomposition by almost affine maps}

We will need some material from David--Semmes \cite{DavidSemmes}. Let $\Delta$ denote the collection of dyadic cubes in $\mathbb{R}^d$, i.e., the collection of all cubes of the form ${[j_12^k, (j_1+1)2^k]} \times  \cdots \times [j_d2^k, (j_d+1)2^k]$ for some $j_1, \ldots, j_d, k \in \mathbb{Z}$.

\begin{definition}[\cite{DavidSemmes}, Definition I.3.13]\label{def:coronization}
A \textit{coronization} of $[0,1]^d$ is a triple $(\cG, \cB, \cF)$, where $\cG$ and $\cB$ are subsets of $\Delta$ and $\cF$ is a family of subsets of $\cG$, satisfying the following:
\begin{enumerate}[(i)]
\item $\Delta = \cG \cup \cB$ and $\cG \cap \cB = \emptyset$.
\item There is a constant $C> 0$ such that $$\sum_{Q\in \cB, Q \subseteq R} |Q| \leq C|R| \text{ for all } R\in\Delta.$$
\item $\cF$ is a collection of mutually disjoint subsets of $\cG$ whose union is $\cG$. Each element $S \in \mathcal{F}$ is called a \textit{stopping time region}.
\item Each $S\in \cF$ is coherent, meaning that it has a unique maximal cube $Q(S)$ containing all its elements, that if a cube is both an ancestor and a descendant of elements of $S$ then it is an element of $S$, and that if $Q\in S$, then either all children of $Q$ are in $S$ or none are.
\item The maximal cubes $\{Q(S): S \in \cF\}$ satisfy
$$ \sum_{S\in\cF, Q(S)\subseteq R} |Q(S)| \leq C|R| \text{ for all } R\in \Delta.$$
\end{enumerate}
The constant $C$ appearing in (ii) and (v) is called the Carleson packing constant of the coronization.
\end{definition}

If $f\colon \RR^d \to \RR^n$ is Lipschitz, then its graph in $\RR^{d+n}$ is a basic example of a \textit{uniformly rectifiable set}, in the language of David--Semmes. It therefore determines a coronization with certain special properties---a \textit{corona decomposition}---in the following way.

\begin{theorem}[David--Semmes \cite{DavidSemmes}]\label{thm:DS}
Let $f\colon [0,1]^d  \rightarrow \RR^n$ be $L$-Lipschitz. Then $\graph(f)$ admits a corona decomposition. That is, for each $\eta, \theta>0$, there is a coronization  $(\cG, \cB, \cF)$ of $[0,1]^d$ with the following property: For each $S\in\mathscr{F}$, there is a $d$-dimensional (rotated) $\eta$-Lipschitz graph $\Gamma=\Gamma(S)$ such that $\dist((x,f(x)),\Gamma(S)) \leq \theta\diam(Q)$ whenever $x\in 2Q$ and $Q\in S$. 

The Carleson packing constant of the coronization depends only on $L$, $d$, $\eta$, and $\theta$.
\end{theorem}
That $\Gamma$ is a $d$-dimensional (rotated) $\eta$-Lipschitz graph means the following: there is an $\eta$-Lipschitz $H\colon \RR^d\rightarrow \RR^n$ and an orthogonal linear map $O\colon \RR^{d+n}\rightarrow\RR^{d+n}$ such that 
\begin{equation}\label{eq:lipgraph}
\Gamma = O(\graph(H)) \subseteq \RR^{d+n}.
\end{equation}
From now on, we will drop the word ``(rotated)'' and just refer to such a $\Gamma$ as a Lipschitz graph.

\begin{remark}
The notion of ``corona decomposition'' defined in \cite{DavidSemmes} (p. 57) applies to sets much more general than graphs, and therefore uses a general notion of ``dyadic cubes'' in an arbitrary Ahlfors regular set. These cubes are defined by properties (3.1)-(3.4) of I.3 (p. 53) of \cite{DavidSemmes}.

Since we are only looking at Lipschitz graphs, our Theorem \ref{thm:DS} is a simplified version of the main result from \cite{DavidSemmes}. In our case, rather than using these general dyadic cubes, we can use the simpler notion defined as follows: Let $f\colon [0,1]^d\rightarrow \RR^n$ be a Lipschitz function. Each dyadic cube $Q\in\Delta$ has a natural counterpart $\hat{Q}=\graph(f|_Q)$ in $\graph(f)\subseteq [0,1]^d \times \RR^n$. The collection of all these $\hat{Q}$ form a collection of generalized ``dyadic cubes'' that are easily seen to satisfy the properties of David--Semmes. Translating the notion of corona decomposition to this language yields Theorem \ref{thm:DS}.
\end{remark}

The goal of this section is to recast Theorem \ref{thm:DS} into a theorem about maps when $f$ is bi-Lipschitz. This yields a new type of corona decomposition for bi-Lipschitz mappings using almost affine mappings, which may have other applications.
\begin{proposition}\label{prop:coronafunction}
Let $f\colon [0,1]^d  \rightarrow \RR^d$ be $L$-bi-Lipschitz. Then, for each $\eta, \theta>0$, there is a coronization  $(\cG, \cB, \cF)$ of $[0,1]^d$ with the following property: For each $S\in\mathscr{F}$, there is a Lipschitz map $g_S\colon \RR^d \rightarrow \RR^d$ such that
\begin{enumerate}[(i)]
\item $g_S$ is $\eta$-almost affine on $Q(S)$.
\item $|g_S(x) - f(x)| \leq \theta\diam(Q)$ for each $Q\in S$ and $x\in 2Q$.
\end{enumerate}
The Carleson packing constant of the coronization depends only on $\eta$, $\theta$, $d$, and $L$.
\end{proposition}

\begin{remark}\label{rmk:gbilip}
    If $f$ is bi-Lipschitz with constant $L\geq 1$, and $\eta,\theta$ are sufficiently small compared to $L$ and $d$, then it follows that each almost affine $g_S=\phi_S\circ A_S$ appearing is $(2L,\eta)$-almost affine and itself $2L$-bi-Lipschitz. 

    For any $S\in\mathcal{F}$, choose $x,y\in 2Q_S$ with $|x-y|\geq \frac{1}{\sqrt{d}}\diam(Q)$. We have
 
    \begin{align*}
    |A_S(x)-A_S(y)| &\leq (1+\eta)|g_S(x)-g_S(y)|\\
    & \leq (1+\eta)(|f(x)-f(y)| + 2\theta\diam(Q))\\
    & \leq (1+\eta)(L|x-y| + 2\theta\sqrt{d}|x-y|)\\
    & = (1+\eta)(L+2\theta\sqrt{d})|x-y|.
    \end{align*}

Similarly,
    \begin{align*}
    |A_S(x)-A_S(y)| &\geq (1+\eta)^{-1}|g_S(x)-g_S(y)|\\
    &\geq (1+\eta)^{-1}(|f(x)-f(y)| - 2\theta\diam(Q))\\
    &\geq (1+\eta)^{-1}(L^{-1}|x-y| - 2\theta\sqrt{d}|x-y|)\\
    &\geq (1+\eta)^{-1}(L^{-1}-2\theta\sqrt{d})|x-y|.
    \end{align*}

This shows that $A$ satisfies appropriate bi-Lipschitz bounds on $2Q$ for points separated by at least $\frac{1}{\sqrt{d}}\diam(Q)$. In particular, this applies to any pair of vertices of $2Q$. Since $A$ is affine, the bounds apply to all pairs of points in $\mathbb{R}^d$.

By choosing $\eta,\theta$ small, we see that $A_S$ and therefore $g_S$ can be made bi-Lipschitz with constant arbitrarily close to $L$.
\end{remark}

We now work towards the proof of Proposition \ref{prop:coronafunction}. 
First, we prove a simple topological lemma, whose proof uses the notion of local degree defined in Section \ref{subsec:localdegree}. 

\begin{lemma}\label{lem:degree}
For each $M, d\geq 1$, there exists $\epsilon=\epsilon(M,d)>0$ with the following property:

If $Q$ is a cube in $\RR^d$ with center $z$ and $F\colon Q\rightarrow\RR^d$ satisfies
$$M^{-1}|x-y|-\epsilon\diam(Q) \leq |F(x)-F(y)| \leq M|x-y|$$
for all $x,y\in Q$, then
$$ F(Q) \supseteq B\left(F(z), \frac{1}{10M}\ell(Q)\right).$$
\end{lemma}
\begin{proof}
By rescaling and translating, it suffices to prove the lemma in the case where $F(z)=z=0$ and $Q$ is the cube centered at $0$ with side length $1$.

Suppose that the lemma (with these normalizations) fails for some $M,d\geq 1$. Then there is a sequence $F_n:Q\rightarrow\RR^d$ satisfying
\begin{equation}\label{eq:Fn}
M^{-1}|x-y|-\frac{1}{n}\diam(Q) \leq |F_n(x)-F_n(y)| \leq M|x-y| \text{ for all } x,y\in Q
\end{equation}
and
\begin{equation}\label{eq:Fn2}
 F_n(Q) \not\supseteq B\left(0,\frac{1}{10M}\right).
\end{equation}

By applying McShane's Theorem (\cite[Theorem 6.2]{Heinonen}) to each coordinate, we may extend each $F_n$ to a Lipschitz map from $\RR^d$ into $\mathbb{R}^d$ that is the identity outside of $2Q$.

Apply the Arzel\'a-Ascoli Theorem to obtain a subsequence of $\{F_n\}$, which we continue to call $\{F_n\}$, that converges uniformly on compact sets to an  $M$-Lipschitz limit map $F\colon \RR^d\rightarrow\RR^d$. Equation \eqref{eq:Fn} implies that $F|_Q$ is $M$-bi-Lipschitz.

We will apply the local degree to the map $F$, with the domain $D$ being the interior of $Q$. Because $F$ is a bi-Lipschitz homeomorphism on $Q$, the local degree $\mu(0, D, F) = \pm 1$. 

Since $F$ is $M$-bi-Lipschitz, the ball $B\left(0,\frac{1}{5M}\right)$ does not intersect $F(\partial D)=F(\partial Q)$. Thus, all of $B\left(0,\frac{1}{5M}\right)$ is in the same component of $\RR^d\setminus F(\partial D)$ as $0$, and so
\begin{equation}\label{eq:degree}
 \mu(y,D,F) = \pm 1 \text{ for all } y\in B\left(0,\frac{1}{5M}\right).
 \end{equation}

Choose $n$ large enough so that $$\sup_{x\in Q}|F(x)-F_n(x)| < \frac{1}{5M}.$$

Consider the straight-line homotopy between  $F$ and $F_n$
$$ H_t(x) = (1-t)F(x) + tF_n(x).$$
It is clear that $H_t$ is proper for each $t$. Moreover, if $y\in B\left(0,\frac{1}{5M}\right)$ and $x\in \partial D = \partial Q$, then
$$ |H_t(x) - y| \geq |F(x)-0| - |y-0| - |F(x)-H_t(x)| \geq \frac{1}{2M} - \frac{1}{5M} - \frac{1}{5M}>0. $$
It follows that $y\notin H_t(\partial D)$ for each $t$.

Thus, homotopy invariance of the local degree implies that
$$ \mu(y, D, F_n) = \pm 1 \text{ for all } y\in  B\left(0,\frac{1}{5M}\right).$$
It follows that $F_n(\overline{D}) = F_n(Q)$ contains all of  $B\left(0,\frac{1}{5M}\right)$, contradicting \eqref{eq:Fn2}.
\end{proof}

Below, we write $\pi_x$ and $\pi_y$ for the orthogonal projections from $\mathbb{R}^{2d}$ to $\RR^d\times\{0\}\cong\RR^d$ and $\{0\}\times \RR^d\cong\RR^d$, respectively. If $V$ is another linear subspace of $\mathbb{R}^{2d}$, we write $\pi_V$ for its associated orthogonal projection.

\begin{lemma}\label{lem:axes}
Let $f\colon [0,1]^d  \rightarrow \RR^d$ be $L$-bi-Lipschitz and $Q\subseteq [0,1]^d$ be a dyadic cube. Let $\Gamma$ be a $d$-dimensional $\eta$-Lipschitz graph, with $H, O$ as in \eqref{eq:lipgraph},  such that $\dist(x,\Gamma) \leq \theta\diam(Q)$ for all $x\in 2\hat{Q}$. Assume that $\eta,\theta$ are sufficiently small depending on $L$ and $d$.

Then 
\begin{equation}\label{eq:axes}
 |\pi_x(O(v))| \approx |\pi_y(O(v))| \approx |v|
\end{equation}
for all $v\in \RR^d \times \{0\}$.  The implied constants depend only on $L$.
\end{lemma}
\begin{proof}
The inequalities $|\pi_x(O(v))|\leq |v|$ and $|\pi_y(O(v))|\leq|v|$ are trivial. It remains to bound $|v|$ from above by the sizes of these projections.

Let 
$$V=O(\RR^d\times\{0\})\subset\RR^d$$ and
$$ \Gamma_0 = \{(x,f(x)): x\in 2Q\}.$$
\begin{claim}\label{claim:projection}
There is a constant  $\epsilon=\epsilon(\theta,\eta)$ that tends to zero with $\theta$ and $\eta$ such that 
$$ |p-q|-\epsilon \leq |\pi_V(p)-\pi_V(q)| \leq |p-q|$$
for all $p,q\in\Gamma_0$.
\end{claim}
\begin{proof}
Fix $p=(x,f(x)),q=(y,f(y))\in 2\hat{Q}$. By our assumption, there are points $p'=O(x',H(x'))$ and $q'=O(y',H(y'))$ in $\Gamma$ such that $|p'-p|$ and $|q'-q|$ are at most $\theta\diam(Q)$ apart.

Because $\Gamma$ is an $\eta$-Lipschitz graph, the map $H$ is $\eta$-Lipschitz. It follows that
$$|x'-y'| \leq |(x',H(x')) - (y',H(y'))| \leq \sqrt{1+\eta^2}|x'-y'|.$$

We therefore have
\begin{align*}
|\pi_V(p) - \pi_V(q)| &\geq |\pi_V(p') - \pi_V(q')| - 2\theta\diam(Q)\\
&= |O(x',0) - O(y',0)| - 2\theta\diam(Q)\\
&= |x'-y'| - 2\theta\diam(Q)\\
&\geq (1+\eta^2)^{-1/2}|(x',H(x')) - (y',H(y'))| - 2\theta\diam(Q)\\
&= |p-q| - (1-(1+\eta^2)^{-1/2}-2\theta)\diam(Q),
\end{align*}
so we may take $\epsilon = (1-(1+\eta^2)^{-1/2}-2\theta)$.

The upper bound in the claim is simply because the projection is $1$-Lipschitz.
\end{proof}

Let $y_0$ denote the center point of $Q$, and define $J\colon [0,1]^d\rightarrow V$ by 
$$ J(x) = \pi_V(x,f(x)).$$ By the previous claim and our assumption on $f$, we have
\begin{equation}\label{eq:Jprops}
(1+L^2)^{-1/2}|x-y| - \epsilon\diam(Q) \leq |J(x)-J(y)| \leq (1+L^2)^{1/2}|x-y|
\end{equation}
for all $x,y\in 2Q$

By choosing $\theta,\eta$ sufficiently small (depending on $L$ and $d$), we may therefore arrange that Lemma \ref{lem:degree} applies to $J$ on $Q$ (identifying the target $V$ of $J$ with $\RR^d$). Thus, for a constant $c=c(M,d)>0$, we get that
\begin{equation}\label{eq:Jball}
J(Q) \supseteq B(J(y_0),c\diam(Q))\cap V.
\end{equation}

Now fix any non-zero $v\in\RR^d\times\{0\}$. As the conclusion of the lemma is unaffected by rescaling $v$, we may assume that $|v|=c\diam(Q)$, where $c$ is the constant from \eqref{eq:Jball}. By \eqref{eq:Jball}, we may find $x\in Q$ such that
$$ J(x) = J(y_0) + O(v).$$
Let $z_0 = (y_0,f(y_0))$ and $w=(x,f(x))$, so that $J(y_0) = \pi_V(z_0)$ and $J(x) = \pi_v(w)$. Choose points $z'_0$ and $w'$ in $\Gamma$ that are at most $\theta\diam(Q)$ away from $z_0$ and $w$, respectively. Observe that, because $\Gamma$ is an $\eta$-Lipschitz graph over $V$,
$$|(z_0'-\pi_V(z_0')) - (w'-\pi_V(w'))| \leq \eta|\pi_V(z_0) - \pi_V(w_0)| \leq (\eta+2\theta)\diam(Q). $$
Thus, we get
\begin{align*}
|O(v) - (w-z_0)| &= |(\pi_V(w) - \pi_V(z_0)) - (w-z_0)|\\
&\leq |(\pi_V(w') - \pi_V(z'_0)) - (w'-z'_0)| +4\theta\diam(Q)\\
&\leq (\eta+6\theta)\diam(Q).
\end{align*}
It follows that
\begin{align*}
|\pi_x(O(v))| &\geq |\pi_x(w-z_0)| - (\eta+6\theta)\diam(Q)\\
&=|x-y_0|- (\eta+6\theta)\diam(Q) \\
&\gtrsim_L |J(x)-J(y_0)| - (\eta+6\theta)\diam(Q)\\
&= |v|-(\eta+6\theta)\diam(Q).
\end{align*}
Since $|v|=c\diam(Q)$, where $c=c(L,d)$, we may take $\eta,\theta<c/2$ to yield
$$|\pi_x(O(v)) \gtrsim_L |v|.$$
For the projection in the $y$-direction, we argue similarly, using the fact that $f$ is bi-Lipschitz:
\begin{align*}
|\pi_y(O(v))| &\geq |\pi_y(w-z_0)| - (\eta+6\theta)\diam(Q)\\
&=|f(x)-f(y_0)|- (\eta+6\theta)\diam(Q) \\
&\gtrsim_L |x-y_0|- (\eta+6\theta)\diam(Q)\\
&\gtrsim_L |J(x)-J(y_0)| - (\eta+6\theta)\diam(Q)\\
&= |v|-(\eta+6\theta)\diam(Q),
\end{align*}
and so again conclude that
$$|\pi_y(O(v)) \gtrsim_L |v|.$$
\end{proof}

The next lemma shows that, under the assumptions of Proposition \ref{prop:coronafunction}, the (rotated) Lipschitz graphs $\Gamma$ are actually unrotated bi-Lipschitz graphs.
\begin{lemma}\label{lem:graph}
Let $f\colon [0,1]^d  \rightarrow \RR^d$ be $L$-bi-Lipschitz. Let $Q\subseteq [0,1]^d$ be a dyadic cube. Let $\Gamma$ be a $d$-dimensional $\eta$-Lipschitz graph, with $H, O$ as above,  such that $\dist(x,\Gamma) \leq \theta\diam(Q)$ for all $x\in 2\hat{Q}$. Assume that $\eta,\theta$ are sufficiently small depending only on $L,d$.

Then there is a bi-Lipschitz $g\colon \RR^d\rightarrow \RR^d$ such that $\Gamma = \graph(g)$. The bi-Lipschitz constant of $g$ can be controlled depending only on $L$.
\end{lemma}
\begin{proof}
Fix two points $p,q\in \Gamma$. Write $V=O(\RR^d\times\{0\})$ as before. Since $\Gamma$ is a Lipschitz graph over $V$,
\begin{align}
|\pi_V(p-q) - (p-q)| &= |\pi_{V^\bot}(p) - \pi_{V^\bot}(q)|\label{eq:perp}\\
&\leq \eta |\pi_V(p) - \pi_V(q)|\nonumber\\
&\leq \eta|p-q|\nonumber
\end{align}
Therefore,
\begin{align*}
|\pi_x(p-q)| &\geq |\pi_x \pi_V(p-q)| - |\pi_x\pi_V(p-q) - \pi_x(p-q)|\\
&\geq |\pi_x \pi_V(p-q)| - |\pi_V(p-q) - (p-q)|\\
&\geq|\pi_x \pi_V(p-q)| -\eta|p-q|  && \text{by \eqref{eq:perp}}\\
&\gtrsim_L |\pi_V(p-q)| - \eta|p-q| && \text{by Lemma \ref{lem:axes}}\\
&\gtrsim_L |p-q| && \text{by \eqref{eq:perp} and $\eta<<1$}
\end{align*}
An identical argument shows that 
$$ |\pi_y(p-q)| \gtrsim_L |p-q|$$
Thus, $\pi_x|_{\Gamma}$ and $\pi_y|_{\Gamma}$ are bi-Lipschitz mappings from $\Gamma$ to $\RR^d$. Since $\Gamma$ is a Lipschitz graph and therefore homeomorphic to $\RR^d$, these maps must be surjective. (Here we make the obvious identifications between $\RR^d\times\{0\}$ and $\RR^d$ and $\{0\}\times\RR^d$ and $\RR^d$.)

Define a map $g\colon \RR^d\rightarrow \RR^d$ by
$$ g = \left(\pi_y|_\Gamma\right) \circ \left(\pi_x|_\Gamma\right)^{-1}.$$
As a composition of bi-Lipschitz maps, $g$ is bi-Lipschitz. 

To see that $\graph(g)=\Gamma$ (without rotation), note that if $p\in \Gamma$, then
$$ p = (\pi_x(p), \pi_y(p)) = (\pi_x(p), g(\pi_x(p))) \in \graph(g).$$
Conversely, if $(x,g(x))\in \graph(g)$, then $x=\pi_x(p)$ for some $p\in \Gamma$ and so $(x,g(x)) =(\pi_x(p), \pi_y(p))= p\in \Gamma$.
\end{proof}

We finally now complete the proof of Proposition \ref{prop:coronafunction}

\begin{proof}[Proof of Proposition \ref{prop:coronafunction}]
Let $f\colon [0,1]^d\rightarrow \RR^d$ be $L$-bi-Lipschitz. Fix $\eta,\theta>0$. Based on these, choose $\eta',\theta'>0$ sufficiently small as specified below. Apply Theorem \ref{thm:DS} to $f$ with parameters $\eta',\theta'$ to obtain a corona decomposition of $\graph(f)$.

Fix $S\in\mathcal{F}$. We obtain a rotated $\eta$-Lipschitz graph $\Gamma$, with associated $H$, $O$ as in \eqref{eq:lipgraph}. By Lemma \ref{lem:graph}, $\Gamma = \graph(g)$ for some $L'$-bi-Lipschitz $g$, where $L'$ depends only on $L$. If $Q\in S$ and $x\in Q$, then $(x,f(x))$ is within distance $\theta'\diam(Q)$ of a point $(z,g(z))\in\Gamma$. It follows that
$$ |f(x) - g(x)| \leq |f(x) - g(z)| + |g(z)-g(x)| \leq \theta'\diam(Q) + L'|z-x| \lesssim_L \theta'\diam(Q).$$
This shows that $g$ approximates $f$ in the way required by Proposition \ref{prop:coronafunction}, if $\theta'$ is sufficiently small depending on $L$ and $\theta$.

Next, we show that $g$ factors in the required way. Let $V = O(\RR^d\times\{0\})$. By Lemma \ref{lem:axes}, $(\pi_x)|_V$ is an invertible linear map from $V$ to $\RR^d\times\{0\} \cong\RR^d$; in fact, it is bi-Lipschitz. Let
$$ A = \pi_y \circ \left((\pi_x)|_V\right)^{-1} : \RR^d \rightarrow \RR^d.$$

By Lemma \ref{lem:axes}, we also have that $(\pi_y)|_V$ is an invertible, indeed bi-Lipschitz, linear map from $V$ to $\{0\}\times\RR^d \cong\RR^d$. Let $\psi\colon \RR^d\rightarrow \RR^d$ be the map 
$$\psi = g \circ \pi_x \circ \left((\pi_y)|_V\right)^{-1}.$$

Then $\psi \circ A = g$. We now explain why $\psi$ is bi-Lipschitz with constant close to $1$. Fix any two points $p_y, q_y\in\mathbb{R}^d$. The reason for these names are that, since $\pi_y|_V$ is surjective by Lemma \ref{lem:axes}, there are points $p=(p_x,p_y),q=(q_x,q_y)\in V\subseteq \RR^d\times\RR^d$.

Let $z=(p_x,g(p_x))\in\Gamma$ and $w=(q_x,g(q_x))\in\Gamma$. Since $z$ and $w$ are on $\Gamma$, so $z = O(t,H(t))$ and $w=O(s,H(s))$ for some $s,t\in\RR^d$.
Note that $\psi(p_y) = \pi_y(z)$ and $\psi(q_y) = \pi_y(w)$.

Let $w' = z + q - p \in \RR^d\times\RR^d$ and $w'' = O(s,H(t))$. Note that
$$w''-w' = O(t-s,0) + q-p \in V.$$
Recall from the proof of Lemma \ref{lem:graph} that $\pi_x|_{\Gamma}$ and $\pi_y|_{\Gamma}$ are bi-Lipschitz (with constant depending only on $L$). We therefore have
\begin{align*}
|w'' - w'| &\approx_L |\pi_x(w'') - \pi_x(w')|&&\text{by Lemma \ref{lem:axes}}\\
&= |\pi_x(w'') - \pi_x(w)|\\
&\leq |w'' - w|\\
&= | O(s,H(t)) - O(s,H(s))|\\
&= |(0,H(t)-H(s))|\\
&\leq \eta'|t-s|\\
&\leq \eta'|z-w|\\
&\lesssim_L \eta'|\pi_y(z) - \pi_y(w)| && \text{as $\pi_y$ is bi-Lipschitz on $\Gamma$}\\
&= \eta'|\psi(p_y) - \psi(q_y)|.
\end{align*}
Observe that this chain also shows along the way that $|w''-w|\lesssim_L  \eta'|\psi(p_y) - \psi(q_y)|$.

Then 
\begin{align}
|\psi(p_y)-\psi(q_y)| &= |g(p_x) - g(q_x)|\nonumber\\
&= |\pi_y(z) - \pi_y(w)|\nonumber\\
&\leq |\pi_y(z) - \pi_y(w')| + |\pi_y(w'-w)|\label{eq:triangle}\\
&\leq |p_y - q_y| + |w'-w''| + |w''-w|\nonumber\\
&\leq |p_y-q_y| + C_L\eta'|\psi(p_y)-\psi(q_y)|\nonumber,
\end{align}
where $C_L$ is the implied constant from the preceding chain of inequalities. If $\eta'$ is sufficiently small depending on $\eta, L$, we get that
\begin{equation}\label{eq:psiupper}
|\psi(p_y)-\psi(q_y)| \leq (1-C_L\eta')^{-1}|p_y-q_y|. 
\end{equation}

By a similar argument (just using the triangle inequality in reverse in \eqref{eq:triangle}), we obtain
$$|\psi(p_y)-\psi(q_y)| \geq |p_y-q_y| - C_L\eta'|\psi(p_y)-\psi(q_y)|$$
and hence
\begin{equation}\label{eq:psilower}
|\psi(p_y)-\psi(q_y)| \geq (1+C_L\eta')^{-1}|p_y-q_y|. 
\end{equation}

If $\eta'$ is sufficiently small (depending on $\eta, L, d$), then Lemma \ref{lem:almostaffine} shows that $g$ can be written as
$$ g = \tilde{\psi}\circ\tilde{A},$$
where $A$ is affine, $\tilde{\psi}$ is $(1+\eta)$-bi-Lipschitz, and
$$ |\tilde{\psi}(x)-x| \leq \eta\diam(Q(S)) \text{ for all }x\in \tilde{A}(Q(S)).$$
This shows that $g$ is $\eta$-almost affine on $Q(S)$, and completes the proof of Proposition \ref{prop:coronafunction}.

\end{proof}

\begin{remark}
There is a somewhat simpler statement closely related to Proposition \ref{prop:coronafunction} that is often useful in geometric measure theory:

\begin{proposition}[David-Semmes \cite{DavidSemmes}, Proposition IV.2.4]\label{prop:linearcoronization}
Let $f\colon \RR^d\rightarrow \RR^k$ be $1$-Lipschitz. For each $\rho, \eta>0$, there is a coronization $(\cB, \cG, \cF)$ of $\RR^d$ and an assignment $Q\mapsto A_Q$ of an affine map $A_Q\colon \RR^d\rightarrow\RR^k$ to each $Q\in\cG$ such that 
$$ \sup_{2Q} |f-A_Q| \leq \rho \diam(Q) \text{ for all } Q\in\cG,$$
and
$$ |A'_Q - A'_{Q(S)}| \leq \eta \text{ whenever } Q\in S\in \cF.$$
The Carleson packing constant can be chosen to depend only on $\rho, \eta, d, k$. 
\end{proposition}

Where Proposition \ref{prop:coronafunction} assigns to each stopping time region $S$ a \textit{single} almost affine map $g_S$, Proposition \ref{prop:linearcoronization} assigns to $S$ a collection of (honest) affine maps $\{A_Q: Q\in S\}$ that approximate $f$ and whose slopes do not change too rapidly. 

Unfortunately, Proposition \ref{prop:linearcoronization} does not appear to suffice for our purposes below. Our argument crucially uses the fact that, for each stopping time region, we have a single map that approximates $f$ well simultaneously on all cubes of the stopping time region. This map will not be affine in general, but Proposition \ref{prop:coronafunction} shows that it can be made almost affine.
\end{remark}

\section{Extension and gluing}\label{sec:extension}

In this section, we give a lemma on extending a bi-Lipschitz map quantitatively from a subset of $\mathbb{R}^d$ to a globally defined map. This lemma is based on a theorem of V\"ais\"al\"a \cite{Vai:86}; see also Theorem III in Azzam--Schul \cite{AS:12}. We also give two lemmas concerning gluing: when a map of $\mathbb{R}^d$ that is bi-Lipschitz on certain subsets of $\mathbb{R}^d$ is globally bi-Lipschitz quantitatively. 

\subsection{Extending from cubes}

Fix an integer $d \geq 2$. A compact set $A \subset \mathbb{R}^d$ has the \textit{bi-Lipschitz extension property} if for every $L'>1$ there exists $L>1$ such that every $L$-bi-Lipschitz embedding $f\colon A \to \mathbb{R}^d$ extends to an $L'$-bi-Lipschitz map $F \colon \mathbb{R}^d \to \mathbb{R}^d$. This is equivalent to the statement that every $L$-bi-Lipschitz map $f\colon A \to \mathbb{R}^d$ extends to an $L'$-bi-Lipschitz map for all sufficiently small $L>1$, where $L' \to 1$ as $L \to 1$. The following theorem was proved by V\"ais\"al\"a \cite[Theorem 5.19]{Vai:86}. 

\begin{theorem} \label{thm:vaisala}
Every set $A \subset \mathbb{R}^d$ that is the union of finitely many compact piecewise linear manifolds of dimension $d-1$ or $d$, with or without boundary, has the bi-Lipschitz extension property.
\end{theorem}

\begin{remark}
In Theorem \ref{thm:vaisala}, for any bi-Lipschitz embedding $f \colon A \to \mathbb{R}^d$, its extension $F$ may be chosen to be piecewise linear outside of $A$. If $f$ itself is piecewise linear, then $F$ may be chosen to be globally piecewise linear. (This final sentence is not stated in \cite{Vai:86}, but it is evident from the proof.)

This remark is not needed here, but it is worth pointing out in light of Corollary \ref{cor:PLapproximation}.
\end{remark}

Of course, the constant $L$ in \ref{thm:vaisala} in general depends on the particular set $A$. We want to show that $L$ can be chosen uniformly under suitable conditions.

\begin{lemma} \label{lemm:vaisala2}
Let $K \geq 1$ and $\Lambda>1$. For all $\epsilon>0$ there exists $\eta>0$ with the following property. Let $Q$ be the image of the unit cube under a $K$-bi-Lipschitz affine map. Let $A = Q \cup \partial(\Lambda Q) $. Then every $(1+\eta)$-bi-Lipschitz embedding $f\colon A \to \mathbb{R}^d$ that is the identity on $\partial(\Lambda Q)$ extends to an $(1+\epsilon)$-bi-Lipschitz map $F\colon \mathbb{R}^d \to  \mathbb{R}^d$ fixing the set $\lambda Q$. The value $\eta$ depends only on $d,K, \Lambda,\epsilon$.

Moreover, let $B = (\mathbb{R}^d \setminus Q^\circ) \cup \partial(\Lambda^{-1} Q)$. Then every $(1+\eta)$-bi-Lipschitz map $f \colon B \to \mathbb{R}^d$ that is the identity on $\partial(\Lambda^{-1}Q)$ extends to an $(1+\epsilon)$-bi-Lipschitz map $F \colon \mathbb{R}^d \to \mathbb{R}^d$ fixing the set $\mathbb{R}^d \setminus (\Lambda^{-1}Q^\circ)$.

\end{lemma}
Before giving the proof, we fix some notation. We identify a linear map from $\mathbb{R}^d$ to itself with its representation in $\mathcal{M}_{d \times d}$, the space of $d \times d$ matrices, with respect to the standard basis. If such a map $S \in \mathcal{M}_{d \times d}$ is invertible, we defined its linear dilatation $H(S)$ as
\[ H(S) = \frac{\sup\{\|Sx\|: \|x\| = 1 \}}{\inf\{\|Sx\|: \|x\| = 1 \}} .\]
We note that $H(S)$ is the quotient of the square root of the largest and smallest eigenvalues of $S^*S$, where $S^*$ is the adjoint of $S$; see \cite[Sec. 14]{Vai:71}. We also observe that $H(S)$ depends continuously on the matrix $S$ and that $H(S) = H(S^{-1})$. 

We define a ``pseudodistance'' $D$ between two maps $S,T \in \GL(d)$ by $D(S,T) = H(S^{-1}T)$. Based on the previous paragraph, we see that $D(\cdot, \cdot)$ is continuous as a function on $\GL(d) \times \GL(d)$. 

\begin{proof}
Let $\mathcal{F}_K$ denote the set of $K$-bi-Lipschitz linear maps of $\mathbb{R}^d$. By translating, we may assume that $Q$ is the image of the unit cube under a map in $\mathcal{F}_K$. By the continuity of the linear dilatation, we see that $\mathcal{F}_K$ is a closed subset of $\GL(d)$. It is also bounded, so we conclude that $\mathcal{F}_K$ is compact.

Consider now a fixed $\epsilon > 0$. For each $S \in \mathcal{F}_K$, by Theorem \ref{thm:vaisala} there exists a constant $\eta(S)$ such that every $(1+\eta(S))$-bi-Lipschitz map $f$ as in the statement of this lemma (with $Q$ the image of $S$) extends to a $\sqrt[3]{1+\epsilon}$-bi-Lipschitz map $F \colon \mathbb{R}^d \to \mathbb{R}^d$. For each $S \in \mathcal{F}_K$, consider the $D$-ball of radius $\sqrt[3]{1+\eta(S)}$ centered at $S$, denoted by $B_S$. Observe that $\sqrt[3]{1+\eta(S)} \leq 1+\eta(S) \leq \sqrt[3]{1+\epsilon}$. Note that, since $D(S,\cdot)$ is a continuous function, the set $B_S$ is open. Since $\mathcal{F}_K$ is compact, we may find a finite set $S_1, \ldots, S_m$ in $\mathcal{F}_K$ such that $\mathcal{F}_K \subset \bigcup_{i=1}^m B_{S_i}$. Now take $\eta = \min_{i=1,\ldots,n}\left\{\sqrt[3]{1+\eta(S_i)}\right\}-1$. Let $T \in \mathcal{F}_K$ be arbitrary and let $Q$ be the image of the unit cube under $T$, and consider a $(1+\eta)$-bi-Lipschitz map $f \colon A \to \mathbb{R}^d$ as in the lemma. There exists $S_i$ such that $T$ is in $B_{S_i}$. We see that $F = S_i \circ T^{-1} \circ f \circ T \circ S_i^{-1}$ is an $(1+\eta(S_i))$-bi-Lipschitz map defined on $Q' \cup \partial(\Lambda Q')$ that is the identity on $\partial(\Lambda Q')$. Here, $Q' = S_iT^{-1}Q$. Hence $F$ extends to a $\sqrt[3]{1+ \epsilon}$-bi-Lipschitz map on $\mathbb{R}^d$, also denoted by $F$. Note that $S_i \circ T$ is $\sqrt[3]{1+\eta(S_i)}$-bi-Lipschitz and hence $\sqrt[3]{1+\epsilon}$-bi-Lipschitz. It follows that the map $T \circ S_i^{-1} \circ F \circ S_i \circ T^{-1}$ is $(1+\epsilon)$-bi-Lipschitz. Finally, it is immediate that $\eta$ depends only on $d,K,\Lambda,\epsilon$, since $\mathcal{F}_K$, the maps $S_i$ and the sets $S_i(Q_0) \cup \partial(\Lambda S_i(Q_0))$ depend only on these quantities. 

The second statement regarding extending a map defined on $B$ can be established similarly. The argument produces a potentially smaller value of $\eta$, so the smaller value of the two satisfies the conclusion of the lemma.

\end{proof}

\subsection{Gluing lemmas}
We state our first lemma on gluing:

\begin{lemma} \label{lemm:pasting}
Let $\mathcal{A}$ be a collection of closed subsets of $\mathbb{R}^d$ for which any two do not intersect except possibly on the boundary. Let $f \colon \mathbb{R}^d \to \mathbb{R}^d$ be a map that is the identity outside $\bigcup_{A \in \mathcal{A}} A^\circ$ and which restricts to an $L$-bi-Lipschitz map of $A$ to itself for all $A \in \mathcal{A}$. Then $f$ is $L^2$-bi-Lipschitz. 

\end{lemma}
\begin{proof}
For each $A \in \mathcal{A}$, define a map $f_A \colon \mathbb{R}^d \to \mathbb{R}^d$ to agree with $f$ on $A$ and be the identity on $\mathbb{R}^d \setminus A$. 
We claim that each $f_A$ is $L$-bi-Lipschitz. Let $x,y \in \mathbb{R}^d$. We want to verify the bi-Lipschitz property for the pair $x,y$. This is immediate if both $x,y \in A$ or both $x,y \in \mathbb{R}^d \setminus A$. In the case that $x \in A$ and $y \in \mathbb{R}^d \setminus A$, we consider the straight line segment from $f(x)$ to $y = f(y)$. This path intersects $\partial A$ at a point $z$. Note that $f_A(z) = z$. Then 
\begin{align*}
    \|f_A(x) - f_A(y)\| = & \|f_A(x) - z\| + \|z-y\|  \\ 
    & \geq (1/L)\|x-z\| + \|z-y\| \\
    & \geq (1/L) (\|x-z| + \|z-y\|) \geq (1/L)\|x-y\|.
\end{align*}
This verifies one inequality. The other inequality can be shown by the same argument applied to $f_A^{-1}$. 

We now verify the lemma. Let $x,y \in \mathbb{R}^d$. Let $A_1, A_2$ be such that $x \in A_1$ and $y \in A_2$. If $x \notin \bigcup \mathcal{A}$ (resp. $y \notin \bigcup \mathcal{A}$), then pick $A_1$ (resp. $A_2$) arbitrarily. Then $f_{A_2} \circ f_{A_1}$ agrees with $f$ on the points $x,y$ and satisfies the bi-Lipschitz condition for $L^2$ on the points $x,y$. 
\end{proof}

We also give the following additional lemma.

\begin{lemma} \label{lem:glue_2}
Let $A_1, A_2$ be closed sets that cover $\mathbb{R}^d$. For each $i \in \{1,2\}$, we are given a $L$-bi-Lipschitz map $f_i$ from $A_i$ to a subset of $\mathbb{R}^d$. Assume these have the property that $f_{A_1} = f_{A_2}$ on the set $A_1 \cap A_2$. Consider the map $f \colon \mathbb{R}^d \to \mathbb{R}^d$ defined by $f|_{A_i} = f_i$. If $f$ is bijective, then $f$ is $L$-bi-Lipschitz. 
\end{lemma}
\begin{proof}
    Let $x \in A_1$,  $y \in A_2$. Let $x' = f(x)$ and $y' = f(y)$. The straight line segment from $x'$ to $y'$ must cross a point $z' \in f(A_1 \cap A_2)$. Then $d(x',y') = d(x',z') + d(z',y') \geq 1/L(d(x,z) + d(z,y)) \geq 1/Ld(x,y)$. The same argument applied to the map $f^{-1}$ shows the reverse inequality. 
\end{proof}

\section{Factoring affine maps locally}\label{sec:factorlinear}

The composition in Theorem \ref{thm:factorization} is built by applying linear maps and translations to individual cubes. In this section, we show how these two types of model maps can be factored into bi-Lipschitz maps with small constant that are the identity away from a given cube.

\subsection{Linear maps}

\begin{lemma}\label{lem:linear}
Let $Q$ be a cube in $\RR^d$ centered at the origin, $A$ an orientation-preserving $L$-bi-Lipschitz linear map from $\RR^d$ to itself, $C>1$, and $\epsilon>0$.
\begin{enumerate}[(i)]
    \item There are $(1+\epsilon)$-bi-Lipschitz maps $h_1, \dots, h_T$ from $\RR^d$ to itself such that
\[h_T \circ \dots \circ h_1 = A \text{ on } Q,\]
and each $h_i$ is the identity outside $CL\sqrt{d}Q$. Moreover, for all $1 \leq i \leq T$, the restriction of $h_i$ to the set $h_{i-1} \circ \cdots \circ h_1(Q)$ is a linear map, and the composition $h_i \circ \cdots \circ h_1|_Q$ is $L$-bi-Lipschitz.
    \item There are $(1+\epsilon)$-bi-Lipschitz maps $\hat{h}_1, \dots, \hat{h}_T$ from $\RR^d$ to itself such that
\[\hat{h}_T \circ \dots \circ \hat{h}_1 = A \text{ on } \mathbb{R}^d \setminus Q\]
and each $\widehat{h}_i$ is the identity inside $\frac{1}{CL}Q$. Moreover, for all $1 \leq i \leq T$, the restriction of $\hat{h}_i$ to the set $\hat{h}_{i-1} \circ \cdots \circ \hat{h}_1(\mathbb{R}^d \setminus Q)$ is a linear map, and the composition $\hat{h}_i \circ \cdots \circ \hat{h}_1|_{\mathbb{R}^d \setminus Q}$ is $L$-bi-Lipschitz.
\end{enumerate}

The number $T$ of mappings can be bounded depending only on $d,L,\epsilon, C$.
\end{lemma}

To prove Lemma \ref{lem:linear}, we first consider the case of diagonal maps (i.e., linear maps represented by a diagonal matrix) and isometries separately. The general case is then obtained using the singular value decomposition and applying a bi-Lipschitz extension result. For a matrix $A$, we let $\|A\|$ denote the usual operator norm of $A$, that is, $\|A\| = \sup\{\|Ax\|: x \in \mathbb{R}^d, \|x\| = 1\}$.

\begin{lemma}\label{lem:diagfactor}
Suppose $D$ is an $L$-bi-Lipschitz diagonal linear map from $\RR^d$ to itself, with non-negative entries. Let $\alpha>0$. Then there are $(1+\alpha)$-bi-Lipschitz diagonal linear maps $D_1, \dots, D_k$ such that
$$ D = D_k \cdots D_1,$$
\begin{equation}\label{eq:diag1}
 \|D_i - I\| < \alpha \text{ for each } i,
 \end{equation}
and
\begin{equation}\label{eq:diag2}
\text{the product }D_j \cdots D_1 \text{ is } L\text{-bi-Lipschitz for each } 1\leq j \leq k
\end{equation}

The number $k$ can be bounded depending only on $d,L,\alpha$.
\end{lemma}
\begin{proof}
Since $D$ is $L$-bi-Lipschitz, the diagonal entries $\sigma_1, \dots, \sigma_d$ of $D$ are bounded between $L^{-1}$ and $L$. Choose $n\in\mathbb{N}$ as small as possible so that 
$$1-\alpha < (1+\alpha)^{-1} < L^{-1/n} \leq L^{1/n} < 1+\alpha.$$

By changing one entry at a time in $n$ steps, we may write $D$ as a product of $dn$ diagonal matrices $D_i$, each of which has diagonal values
\[1,1,\dots,1,\sigma_k^{1/n},1, \dots,1\] 
for some choice of $k\in \{1, \dots, d\}$.

Property \eqref{eq:diag1} follows from our choice of $n$ above. For \eqref{eq:diag2}, notice that, for all $j = 1, \ldots, dn$, the product $D_j\dots D_1$ is a diagonal matrix with all entries between $L^{-1}$ and $L$, and so is $L$-bi-Lipschitz. 
\end{proof}

Next, we consider the case of isometries. 

\begin{lemma}\label{lem:isomfactor}
Let $U$ be an orientation-preserving linear isometry of $\RR^d$ and $\alpha>0$. Then $U$ can be written as a composition of finitely many isometries $U_i$ that are close to the identity in the sense that $\|U_i - I\|<\alpha$ for each $i$. The number of isometries $U_i$ can be bounded depending only on $\alpha$ and $d$.
\end{lemma}
\begin{proof}
Consider $\SO(d)$ as a smooth submanifold of $\RR^{d^2}$ in the usual way. There are (at least) two ways to measure distances between elements of $\SO(d)$. One is to equip $\SO(d)$ with a Riemannian metric inherited from the ambient space, and consider the associated geodesic distance $d_g$. The other is to simply use the ambient distance $d_a$ from $\RR^{d^2}$, which is not Riemannian on $\SO(d)$.

Since $\SO(d)$ is smooth and compact, the two distances $d_a$ and $d_g$ satisfy $d_a \approx_d d_g$. Note also that
$$ d_a(A,B) \approx_d \|A-B\|$$
for all $d\times d$ matrices $A,B$.

If $U\in \SO(d)$, then there is a geodesic $\gamma\colon [0,1]\rightarrow \SO(d)$ (in the metric $d_g$) with $\gamma(0)=I$ and $\gamma(1)=U$. Choose $N\in\mathbb{N}$ and set
$$ \gamma_i = \gamma(i/N).$$
Note that
$$ d_g(\gamma_i, \gamma_{i-1}) = d_g(I,U)/N \lesssim_d 1/N,$$
since $\SO(d)$ is compact.

Define
$$ U_i = \gamma_i \cdot \gamma_{i-1}^{-1} $$
for $i=1,\dots N$. Then each $U_i$ is in $\SO(d)$ and
$$ U = U_N \cdots U_1.$$

For any unit vector $v\in\RR^d$, let $w=\gamma_{i-1}^{-1} v$. Then
\[ |(U_i - I)v| = | \gamma_i\gamma_{i-1}^{-1} v - v| = |\gamma_i w - \gamma_{i-1} w| \leq \|\gamma_i - \gamma_{i-1}\| \]
and hence
\begin{align*}
|(U_i - I)v|
&\leq \|\gamma_i - \gamma_{i-1}\|\\
&\lesssim_d d_a(\gamma_i,\gamma_{i-1})\\
&\lesssim_d d_g(\gamma_i,\gamma_{i-1})\\
&\lesssim_d \frac{1}{N}.
\end{align*}
We conclude that $\|U_i - I\|<\alpha$ if $N$ is chosen sufficiently large, depending only on $d$ and $\alpha$.\end{proof}

\begin{proof}[Proof of Lemma \ref{lem:linear}]
Using the singular value decomposition, write $A = U D V$, where $\Sigma$ is a diagonal matrix with non-negative entries and $U$ and $V$ are orthogonal, i.e., isometries. Since $A$ is orientation-preserving, $U$ and $V$ can also be assumed to be orientation-preserving.

By Lemma \ref{lemm:vaisala2}, there is a choice of $\epsilon'>0$ such that if $K$ is an $L$-bi-Lipschitz affine image of a cube and $h\colon K \cup \partial(CK) \rightarrow CK$ is $(1+\epsilon')$-bi-Lipschitz and equal to the identity on $\partial(CK)$, then it extends to a $(1+\epsilon/2)$-bi-Lipschitz map of $CK$.

Fix $\alpha>0$ sufficiently small, depending on $\epsilon', C, d,L$. First, apply Lemma \ref{lem:isomfactor} to factor $V$ by isometries 
$$ V = V_n \cdots V_1 $$
with each $\|V_i - I \|<\alpha$.

Then apply Lemma \ref{lem:diagfactor} to write $D = D_{m+k}\dots D_{m+1}$ with each $D_i$ $(1+\epsilon')$-bi-Lispchitz and $\|D_i - I\| < \alpha$ for each $i$.

Finally, apply Lemma \ref{lem:isomfactor} again to $U$ to factor $U$ as 
$$ U = U_{m+k+n} \cdots U_{m+k+1}$$
with each $\|U_i - I\|<\alpha$.

For $1\leq i \leq m+k+n$, let $A_i$ be either $V_i$, $D_i$, or $U_i$, whichever is numbered appropriately. 
Note that, for each $i$, 
$$ \| A_i - I \| <\alpha$$
and 
$$ A_{i-1} A_{i-2} \cdots A_1$$
is $L$-bi-Lipschitz (by \eqref{eq:diag2}).

Define a map $h_i$ on $\RR^d$ by setting $h_i$ to be $A_i$ on 
$$Q_i:=A_{i-1}A_{i-2}\cdots A_1(Q)$$
and the identity outside of 
$$R_i:=A_{i-1}A_{i-2}\cdots A_1(CQ) \subseteq CL\sqrt{d}Q.$$
Note that 
$$ h_i(Q_i) = A_i(Q_i) \subseteq R_i$$
if $\alpha$ is small compared to $C$.

We have
$$ h_{m+k+n} \circ \dots \circ h_1 = A \text{ on } Q.$$

Now suppose that $x\in Q_i$ and $y\in \partial(R_i)$. It follows that
$$ |x-y|\gtrsim_{C,L,d} \diam(Q_i)$$

If $\alpha$ is sufficiently small, we therefore have
$$ |h_i(x) -h_i(y)| \leq |h_i(x) - x| + |x-y| \leq \alpha\diam(Q) + |x-y| \leq (1+\epsilon')|x-y|.$$
Similarly
$$|h_i(x) -h_i(y)| \geq |h_i(x) - x| + |x-y| \geq |x-y| - \alpha\diam(Q) \geq (1+\epsilon')^{-1}|x-y|. $$

Thus, $h_i$ extends to a $(1+\epsilon)$-bi-Lipschitz map from $R_i$ into $\RR^d$. Since $h_i$ is the identity on the boundary of $R_i$, we may extend it further to all of $\RR^d$ by the identity and maintain the bi-Lipschitz bound. This completes the proof of (i).

The second claim is proven similarly to the first, and we indicate only the changes required. The linear maps $A_i$ and sets $Q_i$ are defined in the same way as above. Fix $\epsilon'$ as above.

We set $\hat{h}_i$ to be $A_i$ on $Q_i$ and the identity inside
$$ \hat{R}_i := A_{i-1}A_{i-2}\cdots A_1\left(\frac{1}{C}Q\right) \supseteq \frac{1}{CL}Q. $$
By a similar argument as in case (i), if $\alpha>0$ is sufficiently small, the maps $\hat{h}_i$ and  $\hat{h}_i\circ A_i^{-1}$ are both $(1+\epsilon')$-bi-Lipschitz on $R_{i} \cup \partial(Q_{i})$ and $R_{i-1} \cup \partial(Q_{i-1})$, respectively.

Since $h_i \circ A_i^{-1}$ is the identity on $\partial Q_{i-1}$, extends to a global $(1+\epsilon/2)$-bi-Lipschitz map $\hat{j}_i$ as above. It follows that the map $\hat{h_i}$ can be extended to a global map by setting $\hat{h}_i = j_i \circ A_i$. Since $j_i$ is $(1+\epsilon/2)$-bi-Lipschitz and $A_i$ is $(1+\alpha)$-bi-Lipschitz (as it satisfies $\|A_i-I\|<\alpha$), we see that $\hat{h}_i$ is $(1+\epsilon/2)(1+\alpha) < (1+\epsilon)$-bi-Lipschitz if $\alpha$ is sufficiently small. The extended maps $\hat{h}_i$ therefore satisfy the conditions of claim (ii) of the lemma.
\end{proof}

We also record the following variation on Lemma \ref{lem:linear}. The proof is omitted. 

\begin{lemma} \label{lem:shrink}
    Let $Q \subset \mathbb{R}^d$ be a cube centered at the origin, $\lambda>1$, $\epsilon>0$, and $c \in (0,1)$. Then there are $(1+\epsilon)$-bi-Lipschitz maps $h_1, \ldots, h_T$ from $\mathbb{R}^d$ to itself such that $(h_T \circ \cdots \circ h_1)(x) = cx$ on $Q$ and each $h_i$ is the identity outside $\lambda Q$. The number $T$ of mappings depends only on $d, \lambda, \epsilon, c$.
\end{lemma}

\subsection{Translations}

\begin{lemma} \label{lemm:translation}
    Let $Q$ be a cube and $L,\epsilon>0$ be constants. Assume we are given a curve $\gamma$ with initial point $x(Q)$ and some terminal point $y$, with $\ell(\gamma) \leq C \ell(Q)$. There exists $N = N(C,\epsilon,d)$ and mappings $h_1, \ldots, h_N \colon \mathbb{R}^d \to \mathbb{R}^d$ such that each $h_j$ is $(1+\epsilon)$-bi-Lipschitz and is the identity on all points $z$ satisfying $d(z,|\gamma|) \geq 2\diam(Q)$, and the composite $h_N \circ \cdots \circ h_1$ takes $Q$ to the cube $\mathcal{C}(y,\ell(Q))$.  
\end{lemma}
\begin{proof}
Assume that $\gamma$ is parametrized by arc length, and let $\ell = \ell(\gamma)$. For all $t \in [0, \ell]$, let $Q_t$ denote the cube $\mathcal{C}(\gamma(t),\ell(Q))$. Let $\delta>0$ be sufficiently small so that any translation of $Q$ by at most $\delta$ has a $\epsilon$-bi-Lipschitz extension that is the identity outside $2Q$. The existence of such $\delta$ follows from Lemma \ref{lemm:vaisala2}.

Let $N$ be the smallest integer such that $N\delta \geq \ell$. Extend the definition of $\gamma$ to $[0,N\delta]$ by setting $\gamma(t) = \gamma(\ell)$ for all $t \geq \ell$. For all $j \in \{1, \ldots, N\}$, let $h_j$ denote the map given by Lemma \ref{lemm:vaisala2} that translates the cube $Q_{(j-1)\delta}$ to $Q_{j\delta}$ and is the identity outside the set $2Q_{(j-1)\delta}$. It is clear that $h_j$ satisfies the requirements of the lemma and that $h_N \circ \cdots \circ h_1(Q) = Q_\ell = \mathcal{C}(y,\ell(Q))$.  
\end{proof}

\section{Proof of Theorem \ref{thm:factorization}}\label{sec:mainproof}

A convention we will use in this section is to say that constants may depend on {\it the data} to mean that they depend on the given constants in Theorem \ref{thm:factorization} ($d, L, \epsilon, \delta$) as well as all constants previously chosen up to that point (which will ultimately depend only on $d,L,\epsilon,\delta$).

\subsection{Multi-level decomposition}

Fix $f\colon [0,1]^d \rightarrow\RR^d$ an $L$-bi-Lipschitz embedding, and constants $\epsilon,\delta>0$. We first apply Proposition \ref{prop:coronafunction} to obtain a coronization $(\mathscr{G},\mathscr{B},\mathscr{F})$ of $[0,1]^d$ for the map $f$. In doing this, we take $\eta>0$ the value given by Lemma \ref{lemm:vaisala2} with the given $\epsilon$ and with $\Lambda = 2$, and we take $\theta$ so that $L(\theta+\eta)< 1/(4\sqrt{d}L)$. We may assume that the top cube $Q^0_1=[0,1]^d$ is an element of $\mathscr{B}$; this can be arranged and at worst increases the Carleson constant by $1$.

In the first step of the proof, we pass to a useful multi-level decomposition of $[0,1]^d$:

\begin{lemma}\label{lem:stopping}
Let $\alpha>0$. There is an integer $N\in\mathbb{N}$, and constants $\zeta >0$, $K\in\mathbb{N}$, $\lambda=1-2^{-K}$, and for each $n\in\{1, \dots, N\}$ the following objects:
\begin{itemize}
\item Stopping time regions $\{S^n_k\}_{k\in K_n}$ from $\mathscr{F}$. 
\item Dyadic cubes $\{R^n_i\}_{i\in I_n}$, with mutually disjoint interiors and with each $R^n_i$ in some $S^n_k$.
\item Dyadic cubes $\{Q^n_j\}_{j\in J_n}$ with mutually disjoint interiors, each of which is a minimal cube for one of the stopping time regions $S^n_k$ (i.e., $Q^n_j\in S^n_k$ but none of its descendants are).
\end{itemize}
These objects can be chosen to satisfy the following properties:
\begin{enumerate}[(i)]
\item\label{eq:QinR} For all $n\geq 1$, each $Q^n_j$ is contained in some $R^n_i$ and belongs to the same stopping time region as $R^n_i$.
\item\label{eq:RinQ} For all $n\geq 1$, each $R^n_i$ is strictly contained in some $Q^{n-1}_j$. 
\item\label{eq:Rside} If $R^n_i\subseteq Q^{n-1}_j$, then $$ \zeta \ell(Q^{n-1}_j) \leq \ell(R^n_i) \leq 2^{-K}\ell(Q^{n-1}_j).$$
Note in particular that for each cube $R_i^n$ and ancestor $R_{i'}^{n'}$, $R_i^n$ is either contained in $\lambda R_{i'}^{n'}$ or does not overlap with $\lambda R_{i'}^{n'}$. 
\item\label{eq:Bni}
If we set $$B^n_i=(\lambda R^n_i)\setminus \left(\bigcup_j Q^{n}_j \cup \bigcup_{m=1}^N \bigcup_i (R^m_i \setminus \lambda R^m_i\big) \right), $$
then 
$$ |\cup_{n=1}^N \cup_i B^n_i| \geq 1-\alpha.$$
\end{enumerate}

The constants $N$, $\zeta$, $K$, $\lambda$ depend on $\alpha$, $L$, $d$, and the Carleson packing constant $C$ of the coronization.
\end{lemma}

The union of the sets $B^n_i$ is a ``good'' set for us, and will eventually essentially form the set on which the factorization of Theorem \ref{thm:factorization} agrees with $f$.

\begin{proof}
Recall the families of cubes $\cG, \cB$ and the family of stopping time regions $\cF$ from Definition \ref{def:coronization}, with their associated properties. Recall also our assumption that the unit cube $[0,1]^d\in \cB$.

Let $N,\zeta,K, \lambda = 1-2^{-K}$ be constants to be specified below during the course of the proof.

Let $Q^0_1$ be the unit cube and set $J_0=\{1\}$. Define $\{R^1_i\}_{i\in I_1}$ to be the maximal cubes of $\cG$ that are strictly contained in $Q^0_1$ and satisfy
$$ \zeta \ell(Q^{0}_1) \leq \ell(R^1_i) \leq 2^{-K}\ell(Q^{0}_1).$$

Observe that, by definition, the cubes $\{R^1_i\}$ are mutually disjoint and in $\mathscr{G}$
{
and may be of somewhat varying sizes.
}

Therefore, each $R^1_i$ is contained in a stopping time region $S^1_k$. 
Let $\{Q^1_j\}_{j\in J_1}$ denote the collection of all minimal cubes of all the stopping time regions $\{S^1_k\}$
such that there is some $R^1_i \supseteq Q^1_j$.

With $\{R^1_i\}$ and $\{Q^1_j\}$ so defined, we then repeat the construction inductively. That is, for $n\geq 2$, $\{R^n_i\}_{i\in I_n}$ are the maximal sub-cubes of any $Q^{n-1}_j$ that are contained in $\cG$ and satisfy
$$ \zeta \ell(Q^{n-1}_1) \leq \ell(R^n_i) \leq 2^{-K}\ell(Q^{n-1}_1).$$
Each $R^n_i$ is contained in a stopping time region $S^n_k$. We define $\{Q^n_j\}_{j\in J_n}$ to be all the minimal cubes of these stopping time regions that are contained in some $R^n_i$. We terminate this process after $N$ steps, for a value of $N$ to be specified below. (Note that stopping time regions with fewer than $K$ generations are ``skipped over'' in this process.)

Properties \eqref{eq:QinR}, \eqref{eq:RinQ}, and \eqref{eq:Rside} are immediate from the construction. It remains to choose $\zeta, N, K$ so that property \eqref{eq:Bni} holds.

\begin{claim}\label{claim:leftover}
For each $n\geq 0$ and $j\in J_n$, we have
$$ |Q^n_j \setminus \cup_{I_{n+1}} R^{n+1}_i| \leq  \frac{C}{\log_2(1/\zeta) - K} |Q^n_j|.$$
\end{claim}
\begin{proof}
Let $S =Q^n_j \setminus \cup_{I_{n+1}} R^{n+1}_i$. Given a point $x\in [0,1]^d$, let $N(x)$ denote the number of cubes in $\cB$ containing $x$ inside $Q^n_j$. Note that if $x\in S$ then every cube in $Q^n_j$ of side length between $\zeta\ell(Q^n_j)$ and $2^{-K}\ell(Q^n_j)$ that contains $x$ is in $\cB$. Therefore, using Definition \ref{def:coronization},
\begin{align*}
    C|Q^n_j| &\geq \sum_{B\in\cB, B\subseteq Q^n_j} |B|\\
    &= \int_{Q^n_j} N(x) \,dx\\
    &\geq  \int_{S} N(x) \,dx\\
    &\geq |S|(\log_2(1/\zeta) - K),
\end{align*}
which proves the claim.
\end{proof}

\begin{claim}\label{claim:carleson}
We have
$$ \sum_{n=1}^N \sum_{i\in I_n} |R^n_i| \leq C$$
and
$$ \sum_{n=1}^N \sum_{j\in J_n} |Q^n_j| \leq C,$$
where $C$ is the Carleson packing constant
\end{claim}
\begin{proof}
Recall that $\{S^n_k\}_{k\in K_n}$ are the stopping time regions containing the cubes $\{R^n_i\}_{i\in I_n}$. Since the cubes $R^n_i$ are mutually disjoint, we have for each $n$ that
$$ \sum_{i\in I_n} |R^n_i| \leq \sum_{k\in K_n} |Q(S^n_k)|,$$
where $Q(S^n_k)$ are the top cubes of the stopping time regions $S^n_k$. Also note that if $n\neq m$, then any $R^n_i$ and $R^m_{i'}$ must be in distinct stopping time regions. It then follows by Definition \ref{def:coronization} that
$$ \sum_{n=1}^n \sum_{i\in I_n} |R^n_i| \leq \sum_{S\in \mathscr{F}}|Q(S)| \leq C.$$

For the second bound, again observe that for fixed $n$, the cubes $\{Q^n_j\}_{j\in J_n}$ are all disjoint, and that if $n\neq m$, then any $Q^n_j$ and $Q^m_{j'}$ must be in distinct stopping time regions. We therefore get as above that
$$ \sum_{j\in J_n} |Q^n_j| \leq \sum_{k\in K_n} |Q(S^n_k)|,$$
and so
$$ \sum_{n=1}^n \sum_{j\in J_n} |Q^n_j| \leq \sum_{S\in \mathscr{F}}|Q(S)| \leq C.$$
\end{proof}

Set 
$$T = Q^0_1 \setminus \cup_{n=1}^N\cup_i B^n_i,$$
so that our goal is to make $|T|<\alpha$.

Note that $T \subseteq T_{1} \cup T_{2} \cup T_3$, where
$$ T_1 = \cup_{n=1}^N\cup_i (R^n_i \setminus \lambda R^n_i), $$
$T_2$ is the set of all points in $Q^0_1$ that are contained in some $Q^n_j$ for each 
{and every} 
$n=1, 2, \dots, N$, and
$$ T_3 = \cup_{n=1}^{N-1} \cup_j (Q^n_j \setminus \cup_i R^{n+1}_i).$$

Using Claim \ref{claim:carleson}, we bound $T_1$ by
$$ |T_1| \lesssim_d (1-\lambda)\sum_{n=1}^N \sum_{i\in I_n} |R^n_i| \leq C2^{-K}.$$

To bound $T_2$, let $\hat{N}(x)$ denote the number of cubes containing a point $x$ that are either in $\cB$ or are top cubes $Q(S_i)$ of some stopping time region. By Definition \ref{def:coronization}, we have $\int_{Q^1_0} \hat{N}(x) \leq C$.

Each $x\in T_2$ has $\hat{N}(x)\geq N$, because it is contained in at least $N$ minimal cubes of different stopping time regions. We therefore have
$$|T_2| \leq \frac{1}{N}\int_{Q^1_0} \hat{N}(x)\,dx \leq \frac{C}{N}.$$

Finally, using Claims \ref{claim:leftover} and \ref{claim:carleson}, we bound $T_3$ by 
$$ |T_3| \leq \sum_{n=1}^N \sum_{j\in J_n} |(Q^n_j \setminus \cup_i R^{n+1}_i)| \leq \frac{C}{\log_2(1/\zeta) - K}.$$

To force $|T| \leq |T_1|+|T_2|+|T_3| < \alpha$, we therefore first choose $K$ and $N$ large enough so that
$$ C2^{-K} < \alpha/3 \text{ and } \frac{C}{N}<\alpha/3.$$
We then choose $\zeta$ depending on $K$ so that
$$\frac{C}{\log_2(1/\zeta) - K} < \alpha/3.$$
\end{proof}

We apply Lemma \ref{lem:stopping} with parameter $\alpha=\delta/2$. Once these cubes have been chosen, Proposition \ref{prop:coronafunction} gives each $R^n_i$ an almost affine mapping $g^n_i$ associated to the stopping time region $S^n_k$ containing it.

We may apply Lemma \ref{lem:almostaffine} to each $g^n_i$ to write it as $g^n_i = \phi^n_i \circ A^n_i$, where $\phi^n_i$ is $(1+\eta)$-bi-Lipschitz, $A^n_i$ is affine,
$$ |\phi^n_i(x)-x| \lesssim_d (L\eta+\theta)\ell(R^n_i) \text{ on } A^n_i(R^n_i),$$
and
$$ A^n_i(x(R^n_i)) = f(x(R^n_i)).$$

\subsection{Construction of the map $g$}

In this section, we define a map $g$ that agrees with $f$ on a large set and can be factored into $(1+\epsilon)$-bi-Lipschitz maps, thus proving Theorem \ref{thm:factorization}. We fix the following notation.

\begin{itemize}
    \item $I$ is a multiindex $I = (i_1, j_1, \ldots, i_{k-1},j_{k-1},i_k)$ for some parameter $k$. We call $k$ the \textit{order} of $I$.
    \item $J$ is a multiindex $J = (i_1, j_1, \ldots, i_k,j_k)$ for some parameter $k$. We call $k$ the \textit{order} of $J$.
    \item If $I$ is a multindex as above, we let $J(I)$ denote the parent of $I$: $J(I) = (i_1, j_1, \ldots, i_{k-1},j_{k-1})$. If $k=1$, then $J(I)$ is the empty multiindex. Similarly, if $J$ is a multiindex as above, we let $I(J)$ denote the parent of $J$.
    \item We denote the cubes arising in the multilevel decomposition Lemma \ref{lem:stopping} by $R_I^k$ and $Q_J^k$, with $k$ matching the level of the multiindex. The parameter $k$ is technically redundant since it is part of the multiindex $I$, but we include it for clarity. The cubes are indexed so that $1 \leq i_m \leq K_{i_1,j_1, \ldots, i_{m-1},j_{m-1}}$ and $1 \leq j_m \leq K_{i_1,j_1, \ldots, i_{m-1},j_{m-1},i_m}$ for all $1 \leq m \leq k$, where possibly $K_{i_1,j_1, \ldots, i_{m-1},j_{m-1},i_m}= \infty$. Note on the other hand that $K_{i_1,j_1, \ldots, i_{m-1},j_{m-1}} \leq \kappa$ for a constant $\kappa$ depending on the data; that is, the $i_m$ indices are uniformly bounded. Consistent with this notation, the very top cube $[0,1]^d$ is denoted by $Q^0$ (that is, with an empty subscript). 
    \item The center of the cube $R_I^k$ is denoted by $x_I^k$. The center of the cube $Q_J^k$ is denoted by $y_J^k$. 
    \item Each cube $R_I^k$ comes with an $L$-bi-Lipschitz affine map $A_I^k$ and a $(1+\epsilon')$-bi-Lipschitz map $\phi_I^n$ giving an almost affine map $g_I^k = \phi_I^k \circ A_I^k$. Here, $\epsilon'$ is chosen so that every $(1+\epsilon')$-bi-Lipschitz map in the situation of Lemma \ref{lemm:vaisala2} extends to a $\sqrt{1+\epsilon}$-bi-Lipschitz map.  
    \item For each cube $Q_J^k$, $k \geq 1$, we let $\widetilde{Q}_J^k$ denote the closure of $\text{int}(Q_J^k) \cap (\lambda R_{I(J)}^k)$. We also let $\widetilde{Q}^0 = Q^0$. By Lemma \ref{lem:stopping}(\ref{eq:Rside}), $\widetilde{Q}_J^k$ is the empty set or is a rectangle of aspect ratio at most $2$. Moreover, the center $y_J^k$ of $Q_J^k$ is contained in $\widetilde{Q}_J^k$; we continue to refer to $y_J^k$ as the center of $\widetilde{Q}_J^k$. 
    \item Each cube $R_I^k$ is either contained in $\lambda R_{I'}^{k'}$ for all ancestor cubes $R_{I'}^{k'}$, or there is some ancestor $R_{I'}^{k'}$ for which $R_I^k$ and $\lambda R_{I'}^{k'}$ have no overlap. In the following proof, we ignore any cubes of the latter type along with their descendants. Let $\mathcal{R}$ denote the collection of good cubes, i.e., cubes $R_I^k$ of the first type.
\end{itemize}

\subsubsection{Preliminary lemmas} We start with a basic lemma about rearranging collections of cubes. We say that a map is a \textit{similarity map} if it is a map $f\colon \mathbb{R}^d \to \mathbb{R}^d$ of the form $f(x) = rx + x_0$ for some $r>0$ and $x_0 \in \mathbb{R}^d$. Note that such a map preserves coordinate directions.

\begin{lemma}[Shuffling lemma] \label{lem:shuffling}
    Fix $\ell>0$ and $\mu, C_1 >1$. Let $\Omega$ be the image of $[0,\ell]^d$ under an $L$-bi-Lipschitz map $\psi$, where we assume $L \geq 2$. Let $R_1, \ldots, R_N$ be cubes satisfying $\ell(R_j) \geq \ell/C_1$ for each $1 \leq j\leq n$ such that the enlarged cubes $\mu R_j$ are mutually disjoint and contained in $\Omega$. Let $S_1, \ldots, S_N$ also be cubes satisfying $\ell(S_j) \geq \ell/C_1$ for each $1 \leq j\leq n$ such that the enlarged cubes $\mu S_j$ are mutually disjoint and contained in $\Omega$. Then there is a map $G \colon \mathbb{R}^d \to \mathbb{R}^d$ satisfying the following:
    \begin{enumerate}
        \item $G|_{R_j}$ is a similarity map preserving coordinate directions and taking $R_j$ to $S_j$.
        \item $G$ is the composite of $N_0$ maps which are $(1+\epsilon)$-bi-Lipschitz homeomorphisms, where the number of factors $N_0$ depends only on $d,L,C_1,\mu$.
        \item $G$ is the identity map on $\mathbb{R}^d \setminus \Omega$.
    \end{enumerate}
\end{lemma}

\begin{proof}
    For each $1 \leq j \leq N$, let $x_j$ denote the center of $R_j$ and $y_j$ denote the center of $S_j$. Pick an intermediate point $z_j$ for each $j$ as follows. If $y_j$ is not in $R_k$ for any $1 \leq k \leq N$, then take $z_j = y_j$. Otherwise, $y_j \in R_k$ for some $1 \leq k \leq N$. Set $c_1 = 1/(4\sqrt{d} L C_1)$. Observe that $\ell(c_1 R_k) \leq c_1 (L \sqrt{d}\ell) \leq \ell/(4C_1) \leq \ell(S_j)/4$. This implies that there exists some point in $(2^{-1}S_j) \cap (R_k \setminus c_1 R_k)$. We take $z_j$ to be such a point. In addition, we let 
    {the cube}
    $T_j = \mathcal{C}(z_j, c_1 \ell)$. Since $\ell(S_j) \geq \ell/C_1> 4c_1\ell$, we see that $T_j \subset S_j$. In particular, the sets $T_j, T_{j'}$ are disjoint for all $j \neq j'$. Moreover, since $\ell(R_k) \geq \ell/C_1$ and $z_j \in R_k \setminus c_1R_k$, we see that $T_j$ is disjoint from $c_1R_k$ for all $k$. 

    Next, we choose a suitable path $\gamma_j$ from $x_j$ to $z_j$ as follows. The points $\psi^{-1}(x_j)$ and $\psi^{-1}(z_j)$ are connected by a straight-line path in $[0,1]^d$, which we denote by $\widetilde{\gamma}_j^0$. We let $\widetilde{\gamma}_j = \psi \circ \widetilde{\gamma}_j^0$. We obtain $\gamma_j$ by modifying $\widetilde{\gamma}_j$ to remove possible intersections with other cubes  as follows. Suppose $|\widetilde{\gamma}_j|$ intersects the cube $c_1R_k$ for some $k>j$. For any maximal interval $(a,b)$ such that $\widetilde{\gamma}_j|_{(a,b)}$ is contained in $\text{int}(c_1R_k)$, we redefine $\widetilde{\gamma}_j$ on that interval to traverse a shortest path from $\widetilde{\gamma}_j(a)$ to $\widetilde{\gamma}_j(b)$ in $\partial (c_1R_k)$. Next, suppose $|\widetilde{\gamma}_j|$ intersects the cube $T_j$ for some $k<j$. For any maximal interval $(a,b)$ such that $\widetilde{\gamma}_j|_{(a,b)}$ is contained in $\text{int}(T_j)$, we redefine $\widetilde{\gamma}_j$ on that interval to traverse a shortest path from $\widetilde{\gamma}_j(a)$ to $\widetilde{\gamma}_j(b)$ in $\partial \mathcal{C}(x_j,l(j))$. Note that the sets $T_k$, $\mathcal{C}(x_k,c_1 \ell)$ are mutually disjoint for all $k,k'$ (except possible $j=k=k'$), so there is no overlap between the intervals on which the modifications take place. The resulting path after these modifications is called $\gamma_k$.

    We now give a lower bound on $d(|\gamma_j|, \partial \Omega)$. We observe that $d(x_j, \partial \Omega) \geq \ell/C_1$, and consequently that $d(\psi^{-1}(x_j), \partial ([0,1]^d)) \geq \ell/(C_1L)$. The same lower bound holds for $y_j$. It follows that $d(|\widetilde{\gamma}_j^0|, \partial ([0,1]^d)) \geq \ell/(C_1L)$, and hence that $d(|\widetilde{\gamma}_j|, \partial \Omega) \geq \ell/(C_1L^2)$. Any arc on which $\gamma_j$ deviates from $\widetilde{\gamma}_j$ must be within a set $\partial R_k$ for some $1 \leq k \leq N$. Since $d(R_k, \partial Q) \geq (\mu-1)\ell/C_1$, we conclude that \[d(|\gamma_j|, \partial \Omega) \geq \ell \cdot \min\left\{\frac{1}{C_1^2L^2}, \frac{\mu-1}{C_1}\right\}.\]

    We also give an upper bound on $\ell(\gamma_j)$. Consider first an interval $(a, b)$ on which $\widetilde{\gamma}_j$ was modified to give $\gamma_j$. It is easy to see that $\ell(\gamma_j|_{(a,b)}) \leq d \cdot \|\gamma_j(a) - \gamma_j(b)\| \leq d \ell(\widetilde{\gamma}_j|_{(a,b)})$. From this, we conclude that $\ell(\gamma_j) \leq d \ell(\widetilde{\gamma_j})$. But $\ell(\widetilde{\gamma}_j) \leq L \ell(\widetilde{\gamma}_j^0) = L \|\psi^{-1}(x_j) - \psi^{-1}(z_j)\| \leq L^2 \|x_j - z_j\|$. We conclude that 
    \[\ell(\gamma_j) \leq dL^2 \|x_j - z_j\|.\]   

    We are now ready to define the map $G$, which we do by specifying the sequence of factors $G_1, \ldots, G_N$. The strategy is to rescale each cube $R_j$ sufficiently, then translate each $R_j$ to the new center $z_j$ along the path $\gamma_j$, then translate a second time to the final position centered at $y_j$, then rescale to the correct final size. Observe that necessarily $\ell(R_j) \leq \ell \sqrt{d} L$. Let 
    \[c_2 = \min\left\{\frac{c_1}{3\sqrt{d}C_1L}, \frac{1}{2\sqrt{d}C_1^2L^2}, \frac{\mu-1}{2\sqrt{d}C_1}\right\}.\] 
    We define maps $G_1, G_2, \ldots, G_{M_1} \colon \mathbb{R}^d \to \mathbb{R}^d$ as follows. For each $j \in \{1, \ldots, N\}$, Lemma \ref{lem:shrink} gives a sequence of $\sqrt{1+\epsilon}$-bi-Lipschitz maps taking the cube $R_{j}$ to $\mathcal{C}(x_j, c_2\ell)$ that are the identity outside $\mu R_{j}$. Since $\ell(R_j)/(c_2\ell)$ is bounded by $\sqrt{d}L/c_2$, we obtain a bound on the number of maps in this sequence. By postcomposing with the identity as needed, we may ensure that this sequence contains $M_1$ maps for some $M_1 \in \mathbb{N}$ depending only on $d,L,C_1,\mu$. We define $G_1, G_2, \ldots, G_{M_1}$ on each set $\mu R_{j}$ to agree with the sequence produced by Lemma \ref{lem:shrink} and as the identity map elsewhere. It follows from Lemma \ref{lemm:pasting} that each of the maps $G_1, G_2, \ldots, G_{M_1}$ is $(1+\epsilon)$-bi-Lipschitz. 

    Apply Lemma \ref{lemm:translation} to the path $\gamma_j$, $1 \leq j \leq N$, to get a sequence of $(1+\epsilon)$-bi-Lipschitz maps $G_{j}^{1}, G_{j}^{2}, \ldots, G_{j}^{M_2}$ with the properties given in the lemma taking $\mathcal{C}(x_j,c_2\ell)$ to the cube $\mathcal{C}(z_j,c_2\ell)$. Here, the constant $M_2$ is the least upper bound on the number of factors given by Lemma \ref{lemm:translation}. Some of the maps in this sequence may be taken to be the identity if needed to obtain exactly $M_2$ maps. By Lemma \ref{lemm:translation}, the composite $G_j^{M_2} \circ \cdots \circ G_j^1$ is an isometry on the set $\mathcal{C}(x_j,c_2\ell)$. Moreover, each map also fixes each point $z$ satisfying $d(z,|\gamma_j|) \geq 2 \sqrt{d}c_2\ell$. Using this, we show that each map also is the identity on each set $\mathcal{C}(x_k,c_2\ell)$ (for $k>j$) and $\mathcal{C}(z_k,c_2\ell)$ (for $k<j$). Assume first that $k>j$. The definition of the path $\gamma_j$ guarantees that $\gamma_j$ is disjoint from the interior of $aR_k$. Note that 
    \[d(\mathcal{C}(x_j,c_2\ell), \partial (c_1 R_j)) \geq \frac{c_1\ell}{C_1} - c_2\ell \geq \ell \left(\frac{c_1}{C_1} - c_2 \right) \geq \ell \left(\frac{c_1}{C_1L} - \frac{c_1}{3C_1L} \right) = \frac{2c_1 \ell}{3C_1L}.\]
    But $2c_1\ell/(3C_1L) \geq 2\sqrt{d}c_2\ell$, which shows that the cube $\mathcal{C}(x_j,c_2\ell)$ is fixed as claimed. For the case that $k<j$, we use a similar argument for the cube $T_k= \mathcal{C}(z_k,c_1\ell)$. Since $c_1\ell \geq c_1\ell/C_1$, the cube $T_j$ also satisfies the same inequality.

    Moreover, since $2\sqrt{d}c_2\ell \leq d(|\gamma_j|,\partial \Omega)$, we see that each map is the identity on $\mathbb{R}^2 \setminus \Omega$. 

    We now apply an identical procedure to move each cube $\mathcal{C}(z_j,c_1\ell)$ to the cube $\mathcal{C}(y_j, c_1\ell)$ along a path $\zeta_j$ defined in the same way as $\gamma_j$. This produces maps $\widetilde{G}_j^1, \ldots, \widetilde{G}_j^{M_3}$ for each $1 \leq j \leq n$, for some sufficiently large $M_3$ based on Lemma \ref{lemm:translation}. Finally, we define maps $\widetilde{G}_1, \ldots, \widetilde{G}_{M_4}$ similarly to the maps $G_1, \ldots, G_{M_2}$ for some sufficiently large $M_4$ based on Lemma \ref{lem:shrink} that rescales each cube $\mathcal{C}(y_j,a\ell)$ to $S_j$ within the set $\mu S_j$. 

    The final map is $G = \widetilde{G}^{M_4} \circ \cdots \circ \widetilde{G}_1 \circ \widetilde{G}_N^{M_3} \circ \cdots \widetilde{G}_N^1 \circ \cdots \circ \widetilde{G}_1^{M_3} \circ \cdots \circ \widetilde{G}_1^1 \circ G_N^{M_2} \circ \cdots G_N^1 \circ \cdots \circ G_1^{M_2} \circ \cdots \circ G_1^1 \circ G^{M_1} \circ \cdots \circ G_1$. It is evident from the construction that $G$ has the properties stated in the lemma.
\end{proof}

We also have the following lemma.

\begin{lemma} \label{lem:RIk_map}
    Consider a cube $R_I^k$. There is a map $G \colon \mathbb{R}^d \to \mathbb{R}^d$ satisfying the following.
    \begin{enumerate}
        \item $G$ is the composite of $\sqrt{1+\epsilon}$-bi-Lipschitz homeomorphisms $G_1, \ldots, G_{N_0}$, where the number of factors $N_0$ depends only on $d,L,\epsilon$. 
        \item $G = f- f(x_I^k)$ on the set $B_I^k$ 
        {(defined in item \eqref{eq:Bni} of Lemma \ref{lem:stopping}).}
        \item Each $G_i$ is the identity map outside the set $(2L\sqrt{d})R_I^k$.
        \item For each child cube $Q_J^k$ and all $1 \leq i \leq N_0$, $G_i \circ \cdots \circ G_1$ acts as a translation map on the set $(G_{i-1} \circ \cdots \circ G_1)((C_2L\sqrt{d})^{-1}\widetilde{Q}_J^k)$, where $C_2 \geq 2$ is a constant depending only on $L,d$.  
    \end{enumerate}
\end{lemma}
\begin{proof}
    Without loss of generality, assume that $R_I^k$ is centered at the origin and that $f$ fixes $x_I^k = 0$. First, we use Lemma \ref{lem:linear} to produce a sequence of $\sqrt{1+\epsilon}$-bi-Lipschitz maps $h_1, h_{T_0}$ such that $h_{T_0} \circ \cdots h_1 = A_I^k$ on $R_I^k$ and each $h_i$ is the identity outside $2L\sqrt{d}R_I^k$ and also fixes the origin. The restriction of $h_i$ to the set $(h_{i-1} \circ \cdots \circ h_1)(Q_I^k)$ is an affine map, which we denote by $A^i$. We also require, as we may, that each $A^i$ is $(1+\eta/2)$-bi-Lipschitz. Define $h_{T_0+1}\colon \mathbb{R}^d \to \mathbb{R}^d$ to be a $\sqrt{1+\epsilon}$-bi-Lipschitz map that equals $\phi_I^k$ on $A_I^k(R_I^k)$ and is the identity outside $2A_I^k(R_I^k)$. The existence of such a map $h_{T_0+1}$ follows from Lemma \ref{lemm:vaisala2} together with the choice of $\eta$ made earlier. Observe that $h_{T_0+1}$ is the identity outside the set $(2L\sqrt{d})R_I^k$. Moreover, the restriction to $A^k_I(R^k_I)$ of $h_{T_0+1}$ is $(1+\eta)$-bi-Lipschitz. 

    Consider now a child cube $Q_J^k$ of $R_I^k$. For each $i$, we define a map $\widetilde{A}_J^i$ inductively as follows. 
    \begin{itemize}
        \item For $x \in \mathbb{R}^d \setminus (A^{i-1} \circ \cdots \circ A^1)(Q_J^k)$, we set $\widetilde{A}_J^i(x) = A^i(x)$.
        \item For $x \in (A_J^{i-1} \circ \cdots A_J^1)((2L\sqrt{d})^{-1}Q_J^k)$, we set \[\widetilde{A}_J^i(x) = x + w_J^i,\]
        where $w_J^i = (A^i\circ \cdots \circ A^1)(x_J^k) - (A^{i-1}\circ \cdots \circ A^1)(x_J^k)$.
        That is, the point $x$ gets translated by the same amount as $(A^{i-1}\circ \cdots \circ A^1)(x_J^k)$ under the map $A^i$.
        \item Consider the map $\widetilde{A}_J^i$ defined on \[(A_J^{i-1} \circ \cdots \circ A_J^1)((2L\sqrt{d})^{-1}Q_J^k) \cup (\mathbb{R}^d \setminus (A_J^{i-1} \circ \cdots A_J^1)(Q_J^k)).\]

        This map is $(1+\eta)$-bi-Lipschitz, as we now verify. Let $x,y$ be points in this set. If $x,y \in (A_J^{i-1} \circ \cdots \circ A_J^1)((C_2L\sqrt{d})^{-1}Q_J^k)$, then $\|\widetilde{A}_J^i(x) - \widetilde{A}_J^i(y)\| = \|x-y\|$. If $x,y \in \mathbb{R}^d \setminus (A_J^{i-1} \circ \cdots A_J^1)(Q_J^k)$, then $\|\widetilde{A}_J^i(x) - \widetilde{A}_J^i(y)\| = \|A^i(x) - A^i(y)\|$, and hence $x,y$ satisfy the $(1+\eta/2)$-bi-Lipschitz inequality. The last and main case is when $x \in (A_J^{i-1} \circ \cdots \circ A_J^1)((C_2L\sqrt{d})^{-1}Q_J^k)$ and $y \in \mathbb{R}^d \setminus (A_J^{i-1} \circ \cdots A_J^1)(Q_J^k)$. We combine the following inequalities. Since each $\widetilde{A}_J^l$ is on isometry on $(\widetilde{A}_J^{l-1} \circ \cdots \circ \widetilde{A}_J^1)(2L\sqrt{d})^{-1}\ell(Q_J^k)$,
        \[\|(\widetilde{A}_J^i \circ \cdots \circ \widetilde{A}_1^i)(x_J^i) - \widetilde{A}_J^i(x)\| \leq (C_2L\sqrt{d})^{-1}\ell(Q_J^k).\]
        Second, 
        \[ (1+\eta/2)^{-1}\|(\widetilde{A}_J^{i-1} \circ \cdots \circ \widetilde{A}_1^i)(x_J^i) - y\| \leq \|(\widetilde{A}_J^i \circ \cdots \circ \widetilde{A}_1^i)(x_J^i) - \widetilde{A}_J^i(y)\| \leq (1+\eta/2)\|(\widetilde{A}_J^{i-1} \circ \cdots \circ \widetilde{A}_1^i)(x_J^i) - y\|. \]
        Third,
        \[((2L)^{-1} - (C_2L\sqrt{d})^{-1})\ell(Q_J^k) \leq \|x - y\|.\]
        We conclude that 
        \[\|\widetilde{A}_J^i(x) - \widetilde{A}_J^i(y)\| \leq (1+\eta/2)\|x-y\|+ (C_2L\sqrt{d})^{-1}\ell(Q_k)\]
        and 
        \[\|\widetilde{A}_J^i(x) - \widetilde{A}_J^i(y)\| \geq (1+\eta/2)^{-1}\|x-y\| - (C_2L\sqrt{d})^{-1}\ell(Q_k) \]
        By taking $C_2$ large enough, we guarantee that
        \[ (1+\eta)^{-1} \|x-y\| \leq \|\widetilde{A}_J^i(x) - \widetilde{A}_J^i(y)\|  \leq (1+\eta) \|x-y\|.\] 

        According to Lemma \ref{lemm:vaisala2}, $\widetilde{A}_J^i$ extends to a $\sqrt{1+\epsilon}$-bi-Lipschitz map defined on all $\mathbb{R}^d$, which we also denote by $\widetilde{A}_J^i$.
    \end{itemize}
    Finally, we define the map $G_i$ inductively by setting $G_i(x) = \widetilde{A}_J^i(x)$ if $x \in (\widetilde{A}_J^{i-1} \circ \cdots \circ \widetilde{A}_J^1)(Q_J^k)$ for some $J$. Otherwise, we set $G_i(x) = A^i(x)$. Observe that $G_i$ is well-defined since the above sets are disjoint. We check that $G_i$ is $\sqrt{1+\epsilon}$-bi-Lipschitz. Consider two points $x,y \in \mathbb{R}^d$. The only case requiring special consideration is if $x \in (\widetilde{A}_J^{i-1} \circ \cdots \circ \widetilde{A}_J^1)(Q_{J_1}^k)$ but $y \in (\widetilde{A}_J^{i-1} \circ \cdots \circ \widetilde{A}_J^1)(Q_{J_2}^k)$ for some $J_2 \neq J_1$. Consider the map $A_{x,y}^i$ defined by gluing $\widetilde{A}_{J_1}^k$ (defined on the closure of $\mathbb{R}^d \setminus (\widetilde{A}_J^{i-1} \circ \cdots \circ \widetilde{A}_J^1)(R_{J_2}^k)$) and $\widetilde{A}_{J_2}^k$ (defined on the closure of $\mathbb{R}^d \setminus (\widetilde{A}_J^{i-1} \circ \cdots \circ \widetilde{A}_J^1)(R_{J_1}^k)$). It follow from Lemma \ref{lem:glue_2} that $A_{x,y}^i$ is $\sqrt{1+\epsilon}$-bi-Lipschitz. But $A_{x,y}^i(x) = G^i(x)$ and $A_{x,y}^i(y) = G^i(y)$. Thus we get the same conclusion for $G^i$.

    We define the map $G_{T_0+1}$ similarly using $h_{T_0+1}$. 
\end{proof}

\subsubsection{First stage to construct the map $g$: move cubes to their almost affine images, up to a similarity map}

The map $g$ is constructed in two stages. The first stage is contained in the following proposition.

\begin{proposition}\label{prop:firststage}
    There is a map $H \colon \mathbb{R}^d \to \mathbb{R}^d$ satisfying the following.
    \begin{enumerate}
        \item $H$ is the composite of $(1+\epsilon)$-bi-Lipschitz homeomorphisms, where the number of factors depends only on $d,L,\epsilon, \delta$.
        \item For each cube $R_I^k \in \mathcal{R}$, there is a similarity map $s_I^k\colon \mathbb{R}^d \to \mathbb{R}^d$ with scaling factor $r$ satisfying $r_0 \leq r \leq 1$ for some $r_0$ depending only on $d,L,\epsilon, \delta$, such that 
            \[H|_{\lambda R_I^k \setminus \bigcup_J Q_J^k} = s_I^k \circ \phi_I^k \circ A_I^k|_{\lambda R_I^k \setminus \bigcup_J Q_J^k}.\]
        \item The distance between $H(\lambda R_{I_1}^k)$ and $H(\lambda R_{I_2}^k)$, for any $R_{I_1}^{k} \neq R_{I_2}^{k}$, is at least 
        \[c_3 \max\{\ell(R_{I_1}^{k}),\ell(R_{I_2}^{k})\}\]
        for some constant $c_3$ depending only on $d,L,\epsilon, \delta$. 
        \item $H(R_I^k) \subset H(c_1 Q_{J(I)}^{k-1})$, where $c_1 = (C_2 \sqrt{d}L^2)^{-1}$. Here, the constant $C_2$ comes from Lemma \ref{lem:RIk_map}.
        \item $H$ is the identity map outside $B(0, 2L\sqrt{d})$. 
    \end{enumerate}
\end{proposition}
\begin{proof}
    Consider a given cube $Q_J^{k-1}$. It follows from Lemma \ref{lem:stopping} that $Q_J^{k-1}$ has a most $M$ children $R_I^k$ for some $M$ depending only on $d,L,\epsilon,\delta$. Let $c_1 = (C_2\sqrt{d}L^2)^{-1}$, where the constant $C_2$ comes from Lemma \ref{lem:RIk_map}. Let $W_I^{k-1}$ be the set $c_1 Q_J^{k-1} \cap \widetilde{Q}_J^{k-1}$. Then $W_I^{k-1}$ is a rectangle of aspect ratio at most $2$ with shortest side length at least $a_J^{k-1} = 2^{-1}c_1\ell(Q_J^{k-1})$. Pick out a straight line segment in $W_I^{k-1}$ of length $a_J^{k-1}/2$ with distance at least $a_J^{k-1}/4$ from $\partial W_I^{k-1}$ and parallel to a coordinate axis. Take a collection of $M+1$ evenly spaced points on this line segment, so that any two consecutive points are at distance $a_J^{k-1}/(2M)$ apart; associate each cube $R_I^k$ with a distinct point $y_I^k$ from this collection. Let $\ell = a_J^{k-1}(16M \sqrt{d}L)^{-1}$. Apply Lemma \ref{lem:shuffling} to obtain a map $G_J^{k-1}\colon \mathbb{R}^d \to \mathbb{R}^d$ that takes each cube $\lambda R_I^k$ to the cube $\mathcal{C}(y_I^k, \ell)$ and can be factored into $\sqrt{1+\epsilon}$-bi-Lipschitz maps. 
    
    For each $1 \leq k \leq N$, we then define a map $H^k \colon \mathbb{R}^d \to \mathbb{R}^d$ by setting $H^k = G_J^{k-1}$ on each set $Q_J^{k-1}$ and taking $H^k$ to be the identity map elsewhere. The map $H^k$ can be factored into $(1+\epsilon)$-bi-Lipschitz maps by gluing together the respective factorizations of $G_J^{k-1}$ on each set $Q_J^{k-1}$ and applying Lemma \ref{lemm:pasting}. Observe that the restriction of $H^k$ to each set $\lambda R_I^k$ is a similarity map with scaling factor 
    \[ \frac{\ell}{\ell(\lambda R_I^k)} \geq \frac{\ell}{(1-2\lambda) \ell(Q_{J(I)}^{k-1})} = \frac{c_1}{32(1-2\lambda) M \sqrt{d}L}. \]
    Take $\widetilde{r}_0$ to be the right-hand side of the above equality.

    For each cube $R_I^k$, let $\widetilde{R}_I^k = (H^N \circ \cdots \circ H^1)(R_I^k)$ and let $\widetilde{s}_I^k$ be the similarity map taking $R_I^k$ to $\widetilde{R}_I^k$. The scaling factor of $\widetilde{s}_I^k$ is at least $r_0$, where $r_0 = (\widetilde{r}_0)^N$.

    Next, we define maps $\widetilde{H}^1, \ldots, \widetilde{H}^N$ as follows. To define $\widetilde{H}^i$ on the set $(2\sqrt{d}L)\widetilde{R}_I^k$ to be $\widetilde{s}_I^k \circ \widetilde{G}_I^k \circ (\widetilde{s}_I^k)^{-1}$, where $\widetilde{G}_I^k$ is the map from Lemma \ref{lem:RIk_map}. Define $\widetilde{H}^i$ to be the identity map otherwise. By Lemma \ref{lem:RIk_map}, the map $\widetilde{H}^i$ factors into $\sqrt{1+\epsilon}$-bi-Lipschitz maps on each set $(2\sqrt{d}L)\widetilde{R}_I^k$. The respective factors paste together to give factorization into globally defined $(1+\epsilon)$-bi-Lipschitz maps. 

    Finally, we set $H = \widetilde{H}^1 \circ \cdots \circ \widetilde{H}^N \circ H^1 \circ H^2 \circ \cdots \circ H^N$. For each cube $R_I^k$, there is a similarity map $s_I^k$ taking $R_I^k$ to $H(R_I^k)$ with scaling factor at least $r_0$. It remains to verify property (3).

    Consider two cubes $R_{I_1}^k$ and $R_{I_2}^k$ that descend from the same $Q_{J}^{k-1}$ and hence the same set $R_I^{k-1}$. Let $\widetilde{s}_I^{k-1}$ denote the similarity map for the cube $R_I^{k-1}$ and $r_I^{k-1}$ its scaling factor. Then the center points $\widetilde{x}_{I_1}^k, \widetilde{x}_{I_2}^k$ of the two cubes $\widetilde{R}_{I_1}^k$ and $\widetilde{R}_{I_2}^k$ satisfy
    \[ \|x_{I_1}^k - x_{I_2}^k\| = r_I^{k-1} \frac{a_J^{k-1}}{2M} \geq \frac{r_I^{k-1} a_J^{k-1} \max\{\ell(R_{I_1}^k), \ell(R_{I_2}^k)\}}{2M}. \]
    Now each $\widetilde{R}_{I_i}^k$ is contained in $B(\widetilde{x}_{I_i}^k, r_I^{k-1} \ell)$, and hence $H(R_{I_i}^k)$ is contained in 
    \[B(\widetilde{x}_{I_i}^k,2r_I^{k-1}\ell \sqrt{d}L) = B\left(\widetilde{x}_{I_i}^k,\frac{r_I^{k-1} a_J^{k-1}}{8M}\right).\]
    It follows that $H(R_{I_1}^k)$ and $H(R_{I_2}^k)$ are separated by at least
    \[\frac{r_I^{k-1} a_J^{k-1}}{4M} \geq \frac{r_0 c_1}{2}\ell(Q_J^{k-1}) \geq \frac{r_0 c_1}{2}\max\{\ell(R_{I_1}^{k}),\ell(R_{I_2}^{k})\}.\]
    Let $c_3 = r_0c_1/2$. 
\end{proof}

\subsubsection{Secondary subdivision}

Ideally, completing the proof would just be a matter of translating and rescaling the sets $H(\lambda R_I^k)$ to move them onto the sets $(\phi_I^k \circ A_I^k)(\lambda R_I^k)$, beginning with the first level $k=1$ and continuing iteratively through $k=N$. However, there are some technical obstacles that need to be dealt with. For example, for cubes $\lambda R_{I_1}^k$ and $\lambda R_{I_2}^k$ of the same level, the sets $(\phi_{I_1}^k \circ A_{I_1}^k)(\lambda R_{I_1}^k)$ and $(\phi_{I_2}^k \circ A_{I_2}^k)(\lambda R_{I_2}^k)$ may overlap. The solution is to cover each set $H(\lambda R_I^k)$ with a collection of small cubes that, after removing a thin collar from each cube and a small subset of cubes, can be moved around more easily. 

For each cube $R_I^k$, we define a collection $\mathcal{U}_I^k$ of small cubes with a thin collar removed that almost cover the set $H(\lambda R_I^k)$.

\begin{lemma} \label{lem:second_subdivision}
    For each cube $R_I^k \in \mathcal{R}$, there is a collection of cubes $\mathcal{U}_I^k$ such that the following hold for each $k$ and each cube $R_I^k$. 
    \begin{enumerate}
        \item Each cube in $\mathcal{U}_I^k$ intersects the set $H(\lambda R_I^k)$.
        \item For all $U \in \mathcal{U}_I^k$, 
        \[ C'c_4 \ell(Q_{J(I)}^{k-1}) \leq \ell(U) \leq c_4 \ell(Q_{J(I)}^{k-1}),\]
        where $c_4$ and $C'$ are small constants depending only on the data. 
        \item For all distinct cubes $U_1,U_2 \in \mathcal{U}_I^k$, $d(U_1,U_2) \geq \mu \ell(Q_{J(I)}^{k-1})$ for some $\mu>0$ depending only on the data.
        \item If $U_1 \in \mathcal{U}_{I_1}^{k_1}$ and $U_2 \in \mathcal{U}_{I_2}^{k_2}$, where $k_2 > k_1$, then either $U_2 \subset U_1$ or the sets $U_1,U_2$ are disjoint. 
        \item
        \[\left|\lambda R_I^k \setminus H^{-1}\left(\bigcup \mathcal{U}_I^k\right)\right| \leq (\delta/2C)\left|R_I^k\right|.
        \] 
        Recall that $C$ is the Carleson packing constant. 
    \end{enumerate}
\end{lemma}
\begin{proof}
    We define the collections $\mathcal{U}_I^k$ as follows. For each $R_I^k \in \mathcal{R}$, start with the set of cubes of the form
    \begin{align} \label{equ:U_cubes}
       \frac{c_4 \ell(Q_{J(I)}^{k-1})}{p^{k-1}}\left( [ i_1,i_1+1] \times \cdots \times [i_d, i_d+1]\right), 
    \end{align}
    where $i_1, \ldots, i_d \in \mathbb{Z}$, $c_4\leq \zeta (1-\lambda)/(2\sqrt{d}L)$ is a small constant depending on the data, to be specified further in Proposition \ref{prop:second_shuffle}, 
    and $p \in \mathbb{N}$ is sufficiently large as specified below in the proof.  Then let $\mathcal{U}_I^k$ be the set of all cubes $(1- 1/p)U$, where $U$ is a cube of the form in \eqref{equ:U_cubes}, intersecting the set $H(\lambda R_i^1)$.

    We now verify the required properties. Property (1) follows immediately from our construction, as do properties (2) and (3). To verify (4), observe first that $1/\ell(Q_{J(I)}^{k-1})$ is also an integer that divides $1/\ell(Q_{J'}^{k-2})$, where $Q_{J'}^{k-2}$ is the parent of $Q_{J(I)}^{k-1}$. It follows that the side length of level $k$ cubes as in \eqref{equ:U_cubes} divides the side length of parent cubes in $\mathcal{U}_{I'}^{k-1}$, where $R_{I'}^{k-1}$ is the parent of $R_I^k$. 
    
    To obtain property (5), we now choose $p$ to be sufficiently large as follows. For a given cube $U \in \mathcal{U}_I^k$, the collar that was removed has volume at most 
    \[\frac{2d}{p}\left|\frac{U}{1-1/p}\right| \leq \frac{3d}{p}|U|.\] 
    Note that
    \[\sum_{U \in \mathcal{U}_I^k} |U| \leq \left|H(Q_{J(I)}^{k-1})\right| \leq C_5\left|Q_{J(I)}^{k-1}\right| \leq \frac{C_5}{\zeta^d}\left|R_{I}^{k}\right| \]
    for some constant $C_5 \geq 1$ depending on the data. 
    Thus we have removed at most 
    \[\frac{3dC_5}{p\zeta^d}\left|R_I^k\right|\]
    from the set $H(\lambda R_I^k)$. Next, observe that $H^{-1}$ is $L_0$-bi-Lipschitz for some constant $L_0 >1$ depending only on the data. It follows that
    \[\left|\lambda R_I^k \setminus H^{-1}\left(\bigcup \mathcal{U}_I^k\right)\right| \leq \frac{3dC_5L_0^d}{p\zeta^d}\left|R_I^k\right|.\]
    By taking $p$ sufficiently large, we guarantee that (5) is satisfied.
\end{proof}

Let $\mathcal{U}^k = \bigcup_I \mathcal{U}_I^k$.

\subsubsection{Second stage to construct the map $g$: move the cubes to their final position} 

In this step, we define maps $H_k \colon \mathbb{R}^d \to \mathbb{R}^d$, $k = 1, \ldots, N$, inductively as described in the following proposition. In rough terms, the map $H_k$ moves cubes in $\mathcal{U}^k$ of level $k$ to their final position while bringing cubes of subsequent levels to some point not too far removed from their final position. The map $g$ is then defined to be the map $H_N \circ \cdots \circ H_1 \circ H$, where $H$ is the map constructed in Proposition \ref{prop:firststage}.

In implementing this, we treat some cubes of $\allU$ differently than others. Writing $B = \cup_{k,I} B^k_I$, and fixed $k$, $I$, let
$$\tilde{\allU}^k_I = \{U\in \allU^k_I : U \cap H(B) \neq \emptyset\}.$$
Cubes of $\allU^k_I\setminus \tilde{\allU}^k_I$ will be part of the ``exceptional set'' $E$ in Theorem \ref{thm:factorization}.

Fix a constant 
$$ c_b = \frac{\zeta(1-\lambda)}{10L\theta}<1.$$
Then set 
$$ \badU^k_I = \{U\in\tilde{\allU}^k_I: H^{-1}(U) \text{ is covered by cubes } Q^k_J \text{ with } I(J)=I, \ell(Q^k_J)\geq c_b \ell(R^k_I)\}.$$ 

Note that if $U\in{\badU}^k_I$ then $U$ is ``well inside'' a single set $H(Q^k_J)$, since by Proposition \ref{prop:firststage} ${H(B \cap Q^k_J)} \subseteq H(c_1 Q^k_{J})$
and so if $U\in{\badU}^k_I$ then for some $J$
\begin{equation}\label{eq:wellinside}
 U \subseteq N_{\diam(U)}(H(c_1 Q^k_{J})) \subseteq H(2c_1Q^k_J)
\end{equation}
if the constant $c_4$ from Lemma \ref{lem:second_subdivision} was chosen sufficiently small, depending on the distortion of $H$ and the constant $c_b$.

We also set
$$ \goodU^k_I = \tilde{\allU}^k_I \setminus \badU^k_I.$$
Thus, if $U\in \goodU^k_I$, then it contains at least one point $H(y)$ with $y$ not contained in any ``large'' cube $Q^k_J$ (``large'' meaning $\ell(Q^k_J)\geq c_b \ell(R^k_I$). It follows that for this $y$,
$$ |\phi^k_I\circ A^k_I(y) - f(y)|\leq \theta c_b \ell(R^k_I).$$
Hence, if $U\in \goodU^k_I$, then
\begin{equation}\label{eq:goodU}
 \sup_{H^{-1}(U)} |\phi^k_I\circ A^k_I - f| \leq 2\theta c_b \ell(R^k_I),
\end{equation}
again assuming that the constant $c_4$ from Lemma \ref{lem:second_subdivision} was chosen sufficiently small, depending on the distortion of $H$ and the constants $\theta,c_b$.

Finally, we define
$$ G = H(Q^0) \setminus \bigcup_{k=1}^N \bigcup_{I: |I|=k} \left( (\cup \allU^k_I \setminus \cup\tilde{\allU}^k_I) \cup (H(\lambda R^k_I) \setminus \cup\allU^k_I) \right).$$
The factorization found in the next proposition agrees with $f$ on the set $H^{-1}(G) \cap B$. 

\begin{proposition} \label{prop:second_shuffle}
    There are maps $H_1, \ldots, H_N\colon \mathbb{R}^d \to \mathbb{R}^d$ satisfying the following for each $k$ and each cube $R_I^k$.     
    \begin{enumerate}
        \item $H_k$ can be factored as a controlled number of $(1+\epsilon)$-bi-Lipschitz maps.
        \item If
        \[x \in H^{-1}(G \setminus \bigcup \badU^k_I) \cap \lambda R_I^k \]
        then $(H_k \circ \cdots \circ H_1 \circ H)(x) = (\phi_I^k \circ A_I^k)(x)$. Moreover, if $x \in B_I^k \cap H^{-1}(G)$, then $(H_{l} \circ \cdots \circ H_1 \circ H)(x) = (\phi_I^k \circ A_I^k)(x)$ for all $l \geq k$.
        \item $(H_{k} \circ \cdots \circ H_1)\left(H(\lambda R_I^{k+1}) \cap G\right) \subset f(Q_{J(I)}^{k})$. Moreover if $q\in (H_{k} \circ \cdots \circ H_1)\left(H(\lambda R_I^{k+1}) \cap G\right)$ then $d(q, \partial(f(Q^k_{J(I)})))\gtrsim \ell(Q^k_{J(I)})$ for some implied constant depending on the data and the distortions of the mappings $H, H_1, \dots, H_{k-1}$.
       
        \item The restriction of $H_{k}$ to $(H_{k-1} \circ \cdots \circ H_1)(U)$ for each cube $U \in \mathcal{U}^k$ satisfying $U \subset G$ is a similarity map with controlled scaling factor.
    \end{enumerate}
\end{proposition}
\begin{proof}
The maps $H_k$ are constructed inductively in $k$. Fix $k \in \{1, \ldots, N\}$, and assume that $H_{k-1}$ is defined with the required properties. (For the case $k=1$, we take $H_0 = H$ and observe that $H$ satisfies the property that $H(\lambda R_I^1 \cap G) \subset f(Q^0)$ and the ``moreover'' conditions by virtue of being bi-Lipschitz.) We define $H_k$ as follows. On each set $Q_{J}^{k-1}$, we define $H_k$ to be the map produced by Lemma \ref{lem:shuffling} with the following assignments. In the role of the set $\Omega$ is $f(Q_{J}^{k-1})$, which is the image of a cube under an $L$-bi-Lipschitz map. In the role of $R_1, \ldots, R_N$, we assign the cubes
\begin{equation}\label{eq:cubestoshuffle}
    \{H_{k-1}\circ \dots \circ H_1(U) :  U \in \tilde{\allU}^k_I, U \subset G, R_I^k \text{ is a child of } Q_J^{k-1}\}.
\end{equation}
Note that the inductive assumption (4) of the current proposition implies that elements of this set are indeed cubes. Indeed, for each $\tilde{U} = H_{k-1}\circ \dots \circ H_1(U)$ from \eqref{eq:cubestoshuffle}, there is a similarity map $t_{\tilde{U}}$ such that $H_{k-1}\circ \dots \circ H_1|_{U} = t_{\tilde{U}}$.

We now describe the ``target'' locations of the cubes in \eqref{eq:cubestoshuffle}, i.e., the cubes playing the roles of $S_1, \dots, S_N$ from Lemma \ref{lem:shuffling}. Let $\tilde{U}= H_{k-1}\circ \dots \circ H_1(U)$ be a cube in \eqref{eq:cubestoshuffle}. There are two possibilities:
\begin{itemize}
    \item If $\tilde{U}\in\goodU^k_I$ for some $I$: In this case, let $s_{\tilde{U}}$ be the similarity map from Proposition \ref{prop:firststage} taking $(\phi_I^k \circ A_I^k)(R_I^k)$ to $H(R_I^k)$. Set $W_{\tilde{U}} = s_{\tilde{U}}^{-1}\circ t_{\tilde{U}}^{-1}(\tilde{U})$.
    \item If $\tilde{U}\in{\badU}^k_I$: Set $W_{\tilde{U}}=\mathcal{C}(f(H^{-1}(x(U)),r\ell(U))$, where $r>0$ is a small parameter to be chosen below, depending on the data and the distortions of $H, H_1, \dots, H_k$ as well.
\end{itemize}
Now each cube $\tilde{U}$ from \eqref{eq:cubestoshuffle} has been assigned a target cube $W_{\tilde{U}}$. To apply Lemma \ref{lem:shuffling}, we verify six conditions, listed with bullet points below. Let $\tilde{U}_1$, $\tilde{U}_2$ be cubes from \eqref{eq:cubestoshuffle}, with associated $W_i = W_{\tilde{U}_i}$.
\begin{itemize}
\item $\ell(\tilde{U}_1) \gtrsim \ell(Q_J^{k-1})$. 

  For this, observe first that $\ell(\tilde{U}_1) \gtrsim \ell(U_1)$, since $H_{k-1} \circ \cdots \circ H_1$ is a similarity map on the parent cube of $U_1$ with controlled scaling factor. 
  Now $\ell(U_1) \geq C'c_4 \ell(Q_J^{k-1})$ by \Cref{lem:second_subdivision}(2).
\item $d(\tilde{U}_1, \tilde{U}_2)\gtrsim \ell(Q^{k-1}_J)\geq\max\{\ell(\tilde{U}_1),\ell(\tilde{U}_2)\}$, with implied constant depending on the previously chosen data (including the distortions of $H, H_1, \dots, H_{k-1}$). In particular, there is a constant $c$ such that $(1+c)\tilde{U}_1$ does not overlap $\tilde{U}_2$.

This is simply because $d(U_1,U_2)\gtrsim \ell(Q^{k-1}_J)$ by Lemma \ref{lem:second_subdivision} and the mappings $H, H_1, \dots, H_{k-1}$ are bi-Lipschitz.

\item $\tilde{U}_1\in f(Q^{k-1}_J)$ and moreover $d(\tilde{U}_1, \partial f(Q_J^{k-1}))\gtrsim \ell(Q^{k-1}_J)$.

This follows from inductive assumption (3) of the present proposition. 

\item $\ell(W_1) \gtrsim \ell(Q_J^{k-1})$.

This follows similarly to the first bullet point.

\item $d(W_1,W_2) \gtrsim \ell(Q^{k-1}_J) \gtrsim \max\{\ell(W_1),\ell(W_2)\}$. 

The second inequality is a consequence of the fact that all mappings defined are bi-Lipschitz; we need only verify the first.

There are a few cases here. If both $U_1$ and $U_2$ are in ${\badU}^k_I$ for their respective $I$, then the factor of $r$ in the definition of $W_{\tilde{U}}$, if chosen sufficiently small depending on the data so far, causes $W_1$ and $W_2$ to be well-separated.

If both $U_1$ and $U_2$ are in $\goodU^k_I$, then they could either come from the same $I$ or different $I$. If the same $I$, then the similarity map $s_{\tilde{U}}^{-1} \circ t_{\tilde{U}}^{-1}$ being applied is the same for both $U_1$ and $U_2$, so the images $W_1$ and $W_2$ remain well-separated. If $U_1$ and $U_2$ come from different indices $I,I'$, then $H^{-1}(U_1)$ and $H^{-1}(U_2)$ are separated by at least $\frac12 (1-\lambda) \zeta \ell(Q^{k-1}_J)$. If $p_1\in W_1$, then $p_1 = \phi^k_I\circ A^k_I(x_1)$ and $p_2 = \phi^k_{I'}\circ A^k_{I'}(x_2)$ for some $x_i\in H^{-1}(U_i)$. Then, using \eqref{eq:goodU} and the definition of $c_b$,
$$|p_1-p_2| \geq |f(x_1)-f(x_2)|-4c_b\theta\ell(Q^{k-1}_J) \gtrsim \ell(Q^{k-1}_J),$$
and this case is complete.

The last case to consider for this bullet point is $U_1\in {\badU}^k_I$ and $U_2\in \goodU^k_{I'}$. If $I\neq I'$, then a very similar argument to that in the previous paragraph shows again that $W_1, W_2$ are well-separated.

If $I=I'$, i.e., $U_1\in {\badU}^k_I$ and $U_2\in \goodU^k_I$ for the same $I$, then we argue as follows: $U_1$ must be inside $H(2c_1Q^k_J)$ with $\ell(Q^k_J)\geq c_b \ell(R^k_I)$, by \eqref{eq:wellinside}. On the other hand $U_2$ cannot be contained in $H(Q^k_J)$ because it is in $\goodU^k_I$. Therefore
$$ d(H^{-1}(U_1), H^{-1}(U_2)) \geq \frac12 \ell(Q^k_J) \geq \frac12 c_b \ell(R^k_I) \geq \frac12 c_b\zeta \ell(Q^{k-1}_J).$$
Now, consider $p_1\in W_1$ and $p_2\in W_2$. Let $x_i$ be the points of $H^{-1}(U_i)$ mapping to $p_i$ under $H_{k-1}\circ\dots H_1\circ H.$ Then, using the definition of $W_{\tilde{U}}$ and \eqref{eq:goodU}, we have
$$ |p_1 - f(x_1)| \leq r\ell(U_1) \text{ and } |p_2 - f(x_2)| \leq 2c_b\theta\ell(R^k_I).$$
Therefore
\begin{align*}
|p_1-p_2| &\geq |f(x_1)-f(x_2)| - r\ell(U_1)-2c_b\theta\ell(R^k_I)\\
&\geq L^{-1}|x_1-x_2| - r\ell(U_1)-2c_b\theta\ell(R^k_I)\\
&\geq (2L)^{-1}c_b \ell(R^k_I) - 3c_b\theta\ell(R^k_I)\\
&\geq c_b\zeta \ell(Q^{k-1}_J)((2L)^{-1} - 3\theta)\\
&\gtrsim \ell(Q^{k-1}_J),
\end{align*}
using our choice of $\theta$.

\item $W_1\in f(Q^{k-1}_J)$ and moreover $d(W_1, \partial f(Q_J^{k-1})) \gtrsim \ell(Q_I^k)$.

There are two cases here. First, suppose that $U_1\in \goodU^k_I$. Let $p_1\in W_1$. Then, by construction, $p_1=\phi^k_I \circ  A^k_I(x_1)$ for $x_1\in H^{-1}(U)$. We have $x_1\in R^k_I \subseteq Q^{k-1}_J$ and moreover $d(x_1, \partial Q^{k-1}_J)\geq \frac12 (1-\lambda) \ell(R^k_I).$ Therefore, $f(x_1)\in f(Q^{k-1}_J)$ and 
$$d(f(x_1), \partial Q^{k-1}_J)\geq \frac12 L^{-1} (1-\lambda) \ell(R^k_I).$$
On the other hand, using \eqref{eq:goodU},
$$ |p_1-f(x_1)| \leq 2\theta c_b\ell(R^k_I) <  \frac14 L^{-1} (1-\lambda) \ell(R^k_I).$$
It follows that $p_1\in f(Q^{k-1}_J)$ and moreover has a lower bound on its distance to the boundary.

The second case is if $U_1\in\badU^k_I$. Again let $p_1\in W_1$. In this case,
$$|p_1-f(x_1)|\leq \ell(W_1)+L\ell(U_1) \leq (r+L)\ell(U_1)< \frac14 L^{-1} (1-\lambda) \ell(R^k_I)$$
and the argument concludes as in the previous case.

\end{itemize}

 This finishes our definition of $H_k$ on the set $Q_J^{k-1}$. We set $H_k$ to be the identity map for all remaining points. By combining the respective factorizations given by Lemma \ref{lem:shuffling}, we obtain a factorization of $H_k$ into $(1+\epsilon)$-bi-Lipschitz maps. 

 We verify the remaining properties for $H_k$. Properties (1), (2), and (4) are direct from the construction. 
It remains to show that property (3) continues to hold at level $k$. This requires showing that
$$ (H_k \circ H_{k-1} \circ \dots \circ H_1)(H(\lambda R^{k+1}_I) \cap G) \subseteq f(Q^{k}_{J(I)})$$
and moreover that all the points in the set on the left-hand side have distance from $\partial f(Q^k_{J(I)})$ that is bounded below by a multiple of $\ell(Q^k_{J(I)})$.

Let $z\in H(\lambda R^{k+1}_I) \cap G$. Then $z\in H(c_1 Q^k_{J(I)})$. Furthermore, $z$ is in a cube $U$ of either $\goodU^k_I$ or ${\badU}^k_I$.

If $z\in U \in \goodU^k_I$, then by construction
$$ (H_k \circ H_{k-1} \circ \dots \circ H_1)(z) = \phi^k_I \circ A^k_I (H^{-1}(z)).$$
We have
$$ |\phi^k_I\circ A^k_I(H^{-1}(z)) - f(H^{-1}(z))| \leq \theta \ell(Q^k_{J(I)}).$$
Write $y$ for the center of $Q^k_{J(I)}$. Then
\begin{align*}
|(H_k \circ H_{k-1} \circ \dots \circ H_1)(z) &- f(y)| = | \phi^k_I \circ A^k_I (H^{-1}(z)) -  f(y)| \\
&\leq  | \phi^k_I \circ A^k_I (H^{-1}(z)) - f (H^{-1}(z))| +  | f (H^{-1}(z)) -  f(y)||\\
&\leq \theta\ell(Q^k_{J(I)}) + Lc_1\ell(Q^k_{J(I)})\\
&< (3L)^{-1} \ell(Q^k_{J(I)}).
\end{align*}
Thus, in this case, $H_k\circ\dots\circ H_1(z)\subseteq f(Q^k_{J(I)})$ with a lower bound on its distance to the boundary.

If $z\in U \in  {\badU}^k_I$, then $Q^k_{J(I)}$ has (for some $J'$)
$$ \ell(Q^k_{J(I)}) \geq c_b \ell(R^{k}_I) \geq c_b \zeta^{-1}Q^{k-1}_{J'} \geq (10L\sqrt{d})^{-2} \ell(U)$$
if the constant $c_4$ was chosen sufficiently small in Lemma \ref{lem:second_subdivision}.

Continue to write $y$ for the center of $Q^k_J$. Then
\begin{align*}
|(H_k \circ H_{k-1} \circ \dots \circ H_1)(z) &- f(y)| \leq
|f(H^{-1}(z)) - f(y)| + 2L\sqrt{d}r\ell(U)\\
&\leq Lc_1\ell(Q^k_{J(I)}) + (5L)^{-1}\ell(Q^k_{J(I)})\\
&\leq (3L)^{-1} \ell(Q^k_J),
\end{align*}
and so again $H_k\circ\dots\circ H_1(z)\subseteq f(Q^k_{J(I)})$ with a lower bound on its distance to the boundary.
\end{proof}

Now we are ready to complete the proof of Theorem \ref{thm:factorization}. As mentioned previously, we define $g$ to be the map $H_N \circ \cdots \circ H_1 \circ H$.

\begin{proof}[Proof of Theorem \ref{thm:factorization}]
Let $Z = H^{-1}(G)$.

By Proposition \ref{prop:second_shuffle}, $g(x) = (\phi_I^k \circ A_I^k)(x)$ for each cube $R_I^k$ and all $x \in B_I^k \cap Z$. It follows immediately from Proposition \ref{prop:coronafunction} that $(\phi_I^k \circ A_I^k)(x) = f(x)$ for all $x \in B_I^k$, and thus $g(x) = f(x)$ for all $x \in B_I^k \cap Z$. Moreover, $g$ factors as the composition of $(1+\epsilon)$-bi-Lipschitz maps. 

To complete the proof, it suffices to verify by Lemma \ref{lem:stopping} (which we applied with $\alpha=\delta/2$) that 
$|B \setminus Z| \leq \delta/2$,
where $B = \cup_{k,I} B^k_I$.

If $x\in B\setminus Z$, there are two possibilities coming from the definition of $G=H(Z)$ above. One is that $H(x)\in \cup\allU^k_I \setminus \cup\tilde{\allU}^k_I$. However, this is actually impossible, since by definition a cube of $\cup\allU^k_I \setminus \cup\tilde{\allU}^k_I$ does not intersect $H(B)$.

The other possibility is that $x\in \lambda R^k_I\setminus H^{-1}(\cup\allU^k_I)$ for some $k,I$. By Lemma \ref{lem:second_subdivision}, the set of all such points, over all choices of $k,I$, has measure at most
\[ \frac{\delta}{2C} \sum_{I,k} |R_I^k| \leq \frac{\delta}{2C} \cdot C|Q^0| = \frac{\delta}{2}.\]
\end{proof}

\begin{remark}\label{rmk:factorization}
Suppose $f\colon [0,1]^d\rightarrow \RR^d$ is $L$-bi-Lipschitz. Then, for any $\epsilon,\delta>0$, the factorization in Theorem \ref{thm:factorization} can be chosen so that
$$ f_i( 0 ) \in B(f(0), 2L\sqrt{d})$$
for each $i$.
\end{remark}

\section{Piecewise affine approximation}\label{sec:PL}

\subsection{Proof of Corollary \ref{cor:PLapproximation}}
In this section, we prove Corollary \ref{cor:PLapproximation}. This will be a straighforward consequence of Proposition \ref{prop:PL}, which we now state.

\begin{proposition}\label{prop:PL}
For each $d\in\mathbb{N}$, there is an $\epsilon_0=\epsilon_0(d)$ with the following property:

Suppose $f\colon \RR^d \rightarrow \RR^d$ is $(1+\epsilon)$-bi-Lipschitz, for some $\epsilon \leq \epsilon_0$. Then for each $\eta>0$, there is a piecewise affine, $(1+2\epsilon)$-bi-Lipschitz homeomorphism $g \colon \RR^d\rightarrow \RR^d$ such that
$$ \sup_{x\in\RR^d}|g(x)-f(x)| \leq \eta.$$
The mapping $g$ may be chosen to be affine on each simplex in a triangulation of $\RR^d$ consisting of congruent simplices of diameter $\eta/4$. In particular, $N(g,[0,1]^d)\lesssim_d \eta^{-d}.$
\end{proposition}

Proposition \ref{prop:PL} may very well be already known, but we did not find a proof in the literature. We provide one in the next subsection based only on some basic topological degree theory. Pieces of the argument are similar to those appearing in \cite{Vai:86, GCD}.

Before proving this proposition, we show how it combines with Theorem \ref{thm:factorization} to yield Corollary \ref{cor:PLapproximation}:

\begin{proof}[Proof of Corollary \ref{cor:PLapproximation}]
Apply Theorem \ref{thm:factorization} with the given value of $\delta$ and $\epsilon=\epsilon_0(d)$ from Proposition \ref{prop:PL}. This provides a set $E\subseteq [0,1]^d$ with $|E|<\delta$ and $(1+\epsilon)$-bi-Lipschitz mappings $f_T, \dots, f_1$ of $\RR^d$ such that 
$$ f = f_T \circ \dots \circ f_1 \text{ on } [0,1]^d\setminus E.$$
Here $T$ depends only on $d, \delta, L$.

We may now use Proposition \ref{prop:PL} to approximate each $f_i$ by a $(1+2\epsilon)$-bi-Lipschitz piecewise affine homeomorphism $g_i$ with arbitrarily small error $\eta'>0$, and set $g=g_T \circ \dots \circ g_1$. If the error $\eta'$ in each approximation is sufficiently small, depending on $T$, then $|g-f|\leq\eta$ on $[0,1]^d\setminus E$. 

The bi-Lipschitz constant of $g$ is bounded by $(1+2\epsilon)^T$, which is controlled by $\delta,d,L$.

To bound $N(g, [0,1]^d)$, recall that the mappings $g_i$ are all $2$-bi-Lipschitz and defined on a fixed triangulation of $\RR^d$ consisting of congruent simplices $\{S:S\in\mathcal{T}\}$ of diameter $\eta'/4$. For each $k\in\{1, \dots, N\}$, the image of each such simplex $S$ under $g_k \circ \dots g_1$ is covered by at most $n=n(d,T)$ simplices from $\mathcal{T}$. It follows that each simplex $S$ can be decomposed into a controlled number of pieces on which $g_T \circ \dots g_1$ is affine, since the number of such simplices in $[0,1]^d$ is bounded depending on $d$ and $\eta'=\eta'(\eta, d, L, \eta)$.
\end{proof}

\subsection{Proof of Proposition \ref{prop:PL}}
The proof of Proposition \ref{prop:PL} proceeds via a few lemmas. Throughout, we will reference the numbered properties (i)-(vi) of the local degree $\mu$ that are stated in subsection \ref{subsec:localdegree}.

For each $d\in\mathbb{N}$, we fix a triangulation of $\mathbb{R}^d$ into a collection of congruent $d$-simplices of diameter $1$, which we call $\mathcal{T}_d$. As usual, we require that different simplices intersect only in a full lower-dimensional face, if at all. The particular triangulation we choose with these requirements does not matter too much, but for concreteness we fix the following method:

Start with the unit cube of $\RR^d$. Partition it into the so-called ``Freudenthal triangulation'', consisting of $d!$ simplices of the form 
\begin{equation}\label{eq:Freudenthal}
 \{ x\in\RR^d : 0\leq x_{\pi(1)} \leq x_{\pi(2)} \leq \dots \leq x_{\pi(d)}\leq 1\},
\end{equation}
where $\pi$ is any permutation of $\{1,\dots, d\}$. Translate this triangulation to each cube in the standard tiling of $\RR^d$ by unit cubes, and finally rescale it so that all the simplices have diameter $1$.

If $T$ is an image of one (and hence any) of the simplices in \eqref{eq:Freudenthal} under an isometry and scaling of $\RR^d$, we will call $T$ a ``Freudenthal simplex''.

\begin{lemma}\label{lem:affine}
For each $d\in\mathbb{N}$, there is an $\epsilon_0=\epsilon_0(d)\in (0,1]$ with the following property:

Suppose $A$ is an affine map on $\RR^d$, $T$ is a Freudenthal simplex, and $\epsilon\leq \epsilon_0$. Assume that there is a $(1+\epsilon)$-bi-Lipschitz, orientation-preserving mapping $f\colon\RR^d\rightarrow\RR^d$ that agrees with $A$ on the vertices of $T$.

Then $A$ is $(1+2\epsilon)$-bi-Lipschitz and orientation-preserving.

\end{lemma}
\begin{proof}
It suffices to find a value of $\epsilon_0$ that works under the additional assumptions that $T$ is the specific Freudenthal simplex
$$ T = \{x\in\RR^d : 0\leq x_1 \leq x_2 \leq \dots \leq x_d \leq 1\}, $$
and $A(0)=f(0)=0$.

To see that this suffices, suppose we have proven the lemma under these assumptions and $T'$ is another Freudenthal simplex. Then there is a affine map $S$, the composition of a scaling and an isometry, with some scaling factor $a$ that sends $T$ to $T'$. If $A=f$ on the vertices of $T'$, then $a^{-1} r \circ A\circ S = a^{-1} r \circ f\circ S$ on the vertices of $T$, where $r$ is either the identity or a reflection, depending on whether $S$ is orientation-preserving or -reversing. Applying the lemma for these maps on $T$ yields the desired result for $A$ on $T'$.

Next, we show that $A$ must be $(1+2\epsilon)$-bi-Lipschitz.  Write $v_1, \dots, v_n$ for the non-zero vertices of $T$, which of course form a basis of $\RR^d$. The fact that $A$ is $(1+\epsilon)$-bi-Lipschitz on $\{0,v_1, \dots ,v_n\}$, and the polarization identity yield, for each $i,j$, that
\begin{align*}
\langle Av_i, Av_j \rangle &= \frac12\left( |Av_i|^2 + |Av_j|^2 - |Av_i - Av_j^2|\right)\\
&\leq \frac12\left((1+\epsilon)^2(|v_i|^2+|v_j|^2) - (1+\epsilon)^{-2}|v_i-v_j|^2\right)\\
&= \langle v_i, v_j \rangle + \frac12\left(((1+\epsilon)^2-1)(|v_i|^2+|v_j|^2) + (1-(1+\epsilon)^{-2})(|v_i-v_j|^2)\right)\\
&\leq (1+3\epsilon)\langle v_i, v_j \rangle.
\end{align*}
A similar calculation yields
$$\langle Av_i, Av_j \rangle \geq (1-3\epsilon)\langle v_i, v_j \rangle. $$
For an arbitrary vector $v=\sum_{i=1}^d a_i v_i\in\RR^d$,
$$ |Av|^2 = \sum_{1\leq i,j \leq d} a_i a_j \langle Av_i, Av_j \rangle,$$
so
$$ (1-3\epsilon)|v|^2 \leq |Av|^2 \leq (1+3\epsilon)|v|^2,$$
from which it follows that $A$ is bi-Lipschitz with constant $1+2\epsilon\geq\max\{(1+3\epsilon)^{1/2}, (1-3\epsilon)^{-1/2}\}$, if $\epsilon$ is sufficiently small.

We now prove the orientation-preserving property of $A$ on $T$, for $\epsilon$ sufficiently small, by a compactness argument. Suppose that this conclusion were to fail. Then there would be a sequence of orientation-preserving $(1+\epsilon_n)$-bi-Lipschitz mappings $f_n$ of $\RR^d$ ($\epsilon_n\rightarrow 0$) and corresponding orientation-reversing linear maps $A_n$ agreeing with $f_n$ on the vertices of $T$. By the Arzel\`a-Ascoli Theorem, we may pass to a subsequence along which $\{f_n\}$ and $\{A_n\}$ both converge uniformly on compact sets, to maps $f$ and $A$, respectively. Then $f$ is an isometry, therefore affine, and so $A=f$ since they agree on the vertices of the simplex. 

Choose a point $q\in \text{int}(T)$. If $n$ is sufficiently large in our subsequence, the domain $D=\text{int}(T)$, the mappings $A_n$ and $f_n$, and the point $p=f_n(q)$ satisfy the assumptions of Lemma \ref{lem:closedegree}. Therefore
$$ \mu(p, D, A_n) = \mu(p, D, f_n).$$
This contradicts the supposition that $A_n$ are orientation-reversing and $f_n$ are orientation-preserving, and thus completes the proof.
\end{proof}

For the remainder of this section, we fix $d\in\mathbb{N}$ and $\epsilon_0\in (0,1]$ as in Lemma \ref{lem:affine}, and a $(1+\epsilon)$-bi-Lipschitz map $f$ with $\epsilon\leq\epsilon_0\leq 1$. Fix any $\eta>0$; our goal is to approximate $f$ up to error $\eta$ by a global PL homeomorphism of $\RR^d$. We also assume, without loss of generality, that $f$ is orientation-preserving.

Let $\mathcal{T}$ denote the Freudenthal triangulation $\mathcal{T}_d$ of $\RR^d$, with all the simplices rescaled to have diameter $\eta/4$. Let $g$ be affine on each simplex, as uniquely determined by the values of $f$ on the vertices. Observe that $g$ is continuous and in fact $(1+2\epsilon)$-Lipschitz, since (by Lemma \ref{lem:affine}) its restriction to each simplex of $\mathcal{T}$ is. It follows immediately that
$$|g(x)-f(x)|\leq \eta \text{ for all } x\in\RR^d,$$
since both mappings are $2$-Lipschitz and they agree on an $\eta/4$-dense subset of $\RR^d$.

\begin{lemma}\label{lem:surjective}
The map $g\colon\RR^d \rightarrow \RR^d$ is surjective.
\end{lemma}
\begin{proof}
First, observe that $f$ is surjective: since $f$ is bi-Lipschitz, $f(\RR^d)$ is both open and closed in $\RR^d$.

Degree theory then shows that a continuous $g$ within bounded distance of $f$, as our map $g$ is, must also be surjective. In particular, fix any $y\in\RR^d$ with $f(x)=y$. Let $B = B(x, 10\eta)$. Then
$$ \mu(y, B, f) = 1,$$
since $f$ is an orientation-preserving homeomorphism.

We have
$$ \sup_{x\in B} |f(x)-g(x)| \leq \eta $$
and
$$ \dist(y, f(\partial B) \cup g(\partial B) ) \geq 10\eta/2 - 2\eta > \eta,$$
since $f$ is $2$-bi-Lipschitz and $g$ is within $\eta$ of $f$. Lemma \ref{lem:closedegree} then implies that
$$ \mu(y,B, g) = 1.$$
From this, it follows that $y\in g(B)$ and hence $g$ is surjective.
\end{proof}

We now work to show that $g$ is injective. Partition $\RR^d$ into two sets: 
$$ U = \cup_{T\in\mathcal{T}} \text{int}(T) $$
and
$$ B = \cup_{T\in\mathcal{T}} \partial T = \RR^d \setminus U. $$

\begin{lemma}\label{lem:closed}
The sets $B$ and $g(B)$ are closed in $\RR^d$.
\end{lemma}
\begin{proof}
That $B$ is closed is immediate from the fact that $U$ is open.

For $g(B)$, consider an arbitrary closed ball $K = \overline{B}(f(x),r)$ in $\RR^d$. The set $K$ can only intersect $g(T)$ for finitely many simplices $T\in\mathcal{T}$. Indeed, $g^{-1}(K) \subseteq N_{2\eta}(f^{-1}(K))$, so it is bounded and can therefore only intersect finitely many simplices of $\mathcal{T}$.

Thus, $K \cap g(B) = \cup_i (K \cap g(\partial T_i))$ is the union of finitely many compact sets, and thus closed. Since $K$ was an arbitrary closed ball in $\RR^d$, $g(B)$ is closed.
\end{proof}

Now set
$$ W = \RR^d \setminus g(B), $$
$$ V = g^{-1}(W)\subseteq U.$$

\begin{lemma}\label{lem:dense}
The sets $V$ and $W$ are open, dense subsets of $\RR^d$.
\end{lemma}
\begin{proof}
$W$ is open because $g(B)$ is closed (Lemma \ref{lem:closed}), and $V$ is open because $g$ is continuous.

For density of $W$, observe that $B$ and therefore $g(B)$ are measure zero sets; here we use that $g$ is Lipschitz. It follows immediately that $W$ is dense in $\RR^d$.

If $V$ were not dense, then there would be an open set $A$ in $\RR^d$ disjoint from $V$. This open set contains an open subset $A'$ lying in the interior of some simplex $T\in\mathcal{T}$, and so $g(A')$ is open (by Lemma \ref{lem:affine} and invariance of domain). But if $A$, hence $A'$, is disjoint from $V$, it follows that $g(A')\subseteq g(A)\subseteq g(B)$, which has measure zero. This is impossible.
\end{proof}

\begin{lemma}\label{lem:Vinjective}
Every point in $W$ has exactly one pre-image in $\RR^d$ under $g$.
\end{lemma}
\begin{proof}
Fix any $z\in W$. Let $x=f^{-1}(z)$. Note that if $g(y)=z$, then
$$ |x-y| \leq 2|f(x)-f(y)| = 2|g(y)-f(y)| \leq 2\eta,$$
and so
\begin{equation}\label{eq:zball}
g^{-1}(z) \subseteq \overline{B}(x,2\eta).
\end{equation}

Let $D$ be a domain formed by taking all simplices of $\mathcal{T}$ that intersect $B(x,10\eta)$, forming their union, and taking the interior of this set. This is a pre-compact domain in $\RR^d$. Observe that if $y\in\partial D$, then $|y-x|\geq 10\eta$ and so
\begin{equation}\label{eq:far}
|g(y)-z| \geq |f(y)-z| - \eta = |f(y) - f(x)| - \eta \geq 4\eta,
\end{equation}
since $f$ is $2$-bi-Lipschitz.

Let $x_i$ enumerate all the pre-images of $z$ in $\RR^d$, which by \eqref{eq:zball} must all be contained in $D$. There are finitely many, since there are at most one in each simplex of $\mathcal{T}$. Since $z\in W$,  each $x_i$ is in $V\subseteq U$, and so is in the interior of a simplex $T_i\in\mathcal{T}$.

Therefore, 
\begin{equation}\label{eq:zpreimage}
 \mu(z, D, g) = \sum_i \mu(z, \text{int}(T_i), g) = \#\{x_i\}.
\end{equation}
since each $\mu(z, \text{int}(T_i), g)$ is $1$ by Lemma \ref{lem:affine} and property \ref{deg:homeo} of local degree.

On the other hand, by \eqref{eq:far}, we have
$$ \sup_{\partial D} |f(x)-g(x)| \leq \eta < \dist(z, f(\partial D) \cup g(\partial D)).$$
So Lemma \ref{lem:closedegree} says that
$$ \mu(z,D,g) = \mu(z,D,f)=1,$$
since $f$ is an orientation-preserving homeomorphism. Combining this with \eqref{eq:zpreimage} yields
$$ \#\{x_i\} = 1,$$
i.e., that $z$ has a unique pre-image.

\end{proof}

\begin{lemma}\label{lem:open}
The map $g$ is an open mapping.
\end{lemma}
\begin{proof}
This lemma follows the proof of \cite[Lemma 5.5]{GCD}, with some modifications.

Suppose that $g$ were not an open mapping. Then there must be a point $x\in\RR^d$ and an open set $A$ containing $x$ such that $y=g(x)$ is not an interior point of $g(A)$.

Let $\mathcal{T}_x$ be the collection of all (finitely many) simplices of $\mathcal{T}$ that contain $x$. The union $ \cup_{T\in\mathcal{T}\setminus \mathcal{T}_x} T$ is a closed set not containing $x$, so it follows that there is an $r>0$ such that the closed ball $\overline{B}(x,r)$ is contained inside $A \cap \cup_{T\in\mathcal{T}_x} T.$

Since $g$ is injective on each simplex of $\mathcal{T}$, the set $\cup_{T\in\mathcal{T}_x} T$ and therefore the ball $\overline{B}(x,r)$ contains no pre-images of $y$ other than $x$. 

We now claim that
$$ \mu(y,B(x,r),g) = 0.$$
Suppose not. Since the compact set $g(\partial B(x,r))$  does not touch $y$, we can find a small open neighborhood $N$ around $y$ that does not intersect $g(\partial B(x,r))$. By the local constancy of local degree, we must have
$$ \mu(y',B(x,r),g) \neq 0 \text{ for all } y'\in N.$$
It follows that $N\subseteq g(B(x,r))$, but this contradicts the fact that $y$ is not an interior point of $g(A)$.

Therefore, 
$$ \mu(y,B(x,r),g) = 0.$$

We will show that this yields a contradiction. Recalling that the set $V$ is dense in $\RR^d$ (Lemma \ref{lem:dense}), we may choose $x'\in V\cap B(x,r)$ such that $y'=g(x')\in W$ is in the same component of $\RR^d\setminus g(\partial B(x,r))$ as $y$ is. The local constancy of local degree (property \eqref{deg:constant}) then says that
$$ \mu(y', B(x,r), g) = \mu(y,B(x,r),g) = 0.$$

On the other hand, since $x'\in V\subseteq U$, there is a small ball $B'\subseteq B(x,r) \cap U$ containing $x'$ on which $g$ is an orientation-preserving homeomorphism. In addition, Lemma \ref{lem:Vinjective} says that $x'$ is the only pre-image of $y'$ under $g$. Therefore, using property \eqref{deg:preimage}
 of local degree,
 $$ \mu(y', B(x,r), g) = \mu(y', B',g) = 1.$$
This is a contradiction.
\end{proof}

\begin{proof}[Proof of Proposition \ref{prop:PL}]
To prove Proposition \ref{prop:PL}, it remains only to show that the map $g$ constructed above is injective and $(1+2\epsilon)$-bi-Lipschitz.

Injectivity follows exactly as in \cite[p. 211]{GCD}: Suppose $g(x)=g(x')$ for some $x,x'\in \RR^d$. Choose small disjoint open sets $A, A'$ containing $x$ and $x'$, respectively. Because $g$ is an open mapping (Lemma \ref{lem:open}), $g(A)\cap g(A')$ is a non-empty open set. It therefore must contain a point $y\in W$, by Lemma \ref{lem:dense}. But then $y$ has at least two pre-images, one in $A$ and one in $A'$, which contradicts Lemma \ref{lem:Vinjective}.

That $g$ is $(1+2\epsilon)$-Lipschitz was noted above. By Lemma \ref{lem:affine}, if $T\in \mathcal{T}$ then $g^{-1}$ is $(1+2\epsilon)$-Lipschitz on $g(T)$. Since the sets $\{g(T):T\in\mathcal{T}\}$ tile $\RR^d$, it follows easily that $g^{-1}$ is globally $(1+2\epsilon)$-Lipschitz and hence that $g$ is globally $(1+2\epsilon)$-bi-Lipschitz.
\end{proof}

\section{Homeomorphisms of the sphere}\label{sec:sphere}

We follow \cite{FletcherMarkovic} and consider the chordal metric $\chi$ on the unit sphere $\mathbb{S}^d\subseteq \RR^{d+1}$. Identifying $\mathbb{S}^d$ with $\RR^d \cup \{\infty\}$, the chordal metric $\chi$ can be expressed as
\begin{equation*}
\chi(x,y) = 
 \begin{cases}
        \frac{|x-y|}{\sqrt{1+|x|^2}\sqrt{1+|y|^2}} & \text{if } x,y\in\RR^d\\
        \frac{1}{\sqrt{1+|x|^2}}, & \text{if } x\in\RR^d, y=\infty\\
        \frac{1}{\sqrt{1+|y|^2}}, & \text{if } y\in\RR^d, x=\infty
\end{cases}
\end{equation*}

We observe that if $g\colon\RR^d\rightarrow\RR^d$ is bi-Lipschitz, then it extends continuously (setting $g(\infty)=\infty$) to a homeomorphism of $\mathbb{S}^d$. The next lemma relates the distortion of $g$ to the distortion of this extension. It is a mild adustment of \cite[Lemma 2.4]{FletcherMarkovic}, and its proof is very similar.

\begin{lemma}\label{lem:stereographic}
For every $\epsilon>0$, there is an $\epsilon'>0$ with the following property. If $g\colon \RR^d\rightarrow\RR^d$ is $(1+\epsilon')$-bi-Lipschitz in the Euclidean metric and $g(0)=0$, then $g\colon\mathbb{S}^d\rightarrow\mathbb{S}^d$ is $(1+\epsilon)$-bi-Lipschitz in the metric $\chi$.
\end{lemma}
\begin{proof}
Suppose that $g\colon\RR^d\rightarrow\RR^d$ fixes the origin and is $(1+\epsilon')$-bi-Lipschitz, with $\epsilon'>0$ to be decided below.

Exactly as in equation (3.1) of \cite{FletcherMarkovic}, we have, for all $x,y\in \mathbb{R}^d$, that
$$\chi(g(x),g(y)) \leq (1+\epsilon')\chi(x,y)\left(\frac{1+|x|^2}{1+|g(x)|^2}\right)^{1/2}\left(\frac{1+|y|^2}{1+|g(y)|^2}\right)^{1/2}$$

Since $g$ is $(1+\epsilon')$-bi-Lipschitz and fixes the origin, we have
$$ \frac{1+|x|^2}{1+|g(x)|^2} \leq \frac{1+|x|^2}{1+(1-\epsilon')|x|^2} = \left(1- \frac{\epsilon'}{1+|x|^2}\right)^{-1} \leq (1-\epsilon')^{-1}.$$
It follows that
$$ \chi(g(x),g(y)) \leq (1+\epsilon')(1-\epsilon')^{-2}\chi(x,y),$$
and so $g\colon\mathbb{S}^d\rightarrow\mathbb{S}^d$ is $(1+\epsilon)$-Lipschitz if $\epsilon'$ is sufficiently small. (The case where $x$ or $y$ is $\infty$ follows by continuity.) Applying the same to $g^{-1}$ finishes the proof of the lemma.
\end{proof}

We also need the following fact. 
\begin{lemma}\label{lem:translationscaling}
Let $v\in \RR^d$ and $a>0$. Let
$$ \tau(x) = x+v \text{ and } \sigma(x) = ax$$
be the associated translation and scaling maps.

Fix $\epsilon>0$. Both $\tau$ and $\sigma$, considered as homeomorphisms of $\mathbb{S}^d$, can be factored as finite compositions of $(1+\epsilon)$-bi-Lipschitz homeomorphisms.

The number of mappings needed depends only on $\epsilon$,  $|v|$ (in the case of $\tau$), and $|a|$ (in the case of $\sigma$).
\end{lemma}
\begin{proof}
It suffices to show that $\tau$ and $\sigma$ are themselves $(1+\epsilon)$-bi-Lipschitz in the  metric $\chi$ if $|v|$ is sufficiently small or $a$ is sufficiently close to $1$. (Then one can factor an arbitrary translation or scaling as a finite sequence of translations or rescalings with small constants in the obvious way.) As above, we may restrict attention to points in $\RR^d$ by continuity.

We have
\begin{equation}\label{eq:translation}
\chi(\tau(x),\tau(y)) = \chi(x,y) \sqrt{\frac{1+|x|^2}{1+|x+v|^2}}\sqrt{\frac{1+|y|^2}{1+|y+v|^2}}.
\end{equation}
We can write
\begin{align*}
\frac{1+|x|^2}{1+|x+v|^2} &= \left(1 + \frac{2x\cdot v + |v|^2}{1+|x|^2} \right)^{-1}\\
&\leq \left(1 - \frac{2|x||v|}{1+|x|^2} \right)^{-1}\\
&\leq \left(1 - |v|\right)^{-1}
\end{align*}
It follows from this and \eqref{eq:translation} that $\tau$ is $(1+\epsilon)$-Lipschitz in the spherical metric if $|v|$ is small compared to $\epsilon$. Applying this to the inverse map proves that $\tau$ is  $(1+\epsilon)$-bi-Lipschitz if $|v|$ is small.

For $\sigma$, we note that
\begin{equation}\label{eq:scaling}
\chi(\sigma(x),\sigma(y)) = \chi(x,y) \sqrt{\frac{a(1+|x|^2)}{1+a^2|x|^2}}\sqrt{\frac{a(1+|y|^2)}{1+a^2|y|^2}}.
\end{equation}
We have
\begin{align*}
\frac{a(1+|x|^2)}{1+a^2|x|^2} &= a \left(1 + (a^2-1)\frac{|x|^2}{1+|x|^2}\right)^{-1}\\
&\leq a\left(1-|a^2-1|\right)^{-1},
\end{align*}
from which it follows that $\sigma$ is $(1+\epsilon)$-Lipschitz in the spherical metric if $a$ is sufficiently close to $1$. As above, this completes the proof.
\end{proof}

\begin{proof}[Proof of Corollary \ref{cor:sphere}]
Let $f\colon \mathbb{S}^d\rightarrow\mathbb{S}^d$ be an $L$-bi-Lipschitz homeomorphism of the sphere, with respect to $\chi$, and fix $\delta,\epsilon>0$. By post-composing $f$ with an isometry of $\mathbb{S}^d$, which can be ``undone'' in the factorization, we may assume that $f(\infty)=\infty$. In this case, we have from the fact that $f$ is $L$-bi-Lipschitz and the definition of $\chi$ that
\begin{equation}\label{eq:f0}
 |f(0)| \leq \sqrt{L^2-1}.
\end{equation}

Choose $R=R(\delta)>0$ sufficiently large so that the $d$-dimensional Hausdorff measure of
$$\mathbb{S}^d\setminus [0,R]^d$$
in the metric $\chi$ is at most $\delta/2$.

Let $F\colon [0,R]^d\rightarrow \RR^d$ be the restriction of the map $f$, viewed as a mapping on $\RR^d$. The Euclidean and spherical metrics are comparable on $[0,R]^d$, with constant $C = C(\delta,d)$. Therefore, $F$ is bi-Lipschitz in the Euclidean metric, with constant $L'=CL = L'(L,\delta,d)$.

Let $G(x) = R^{-1}F(Rx)$, an $L'$-bi-Lipschitz map from $[0,1]^d$ into $\RR^d$. Fix $\epsilon'>0$ as provided by Lemma \ref{lem:stereographic}. By Theorem \ref{thm:factorization}, there are $T=T(\epsilon,\delta,d)$ $(1+\epsilon')$-bi-Lipschitz maps $g_i\colon \RR^d\rightarrow\RR^d$ and a set $A\subseteq [0,1]^d$ such that
$$ G = g_T \circ \dots \circ g_1 \text{ on } [0,1]^d\setminus A$$
and the Euclidean Hausdorff measure of $A$ (which is comparable to Lebesgue measure) satisfies
$$ \mathcal{H}^d(A) < \delta/(2CR)^d.$$

By interposing translations, we may rewrite the above factorization of $G$ as 
$$ G = \tau_T \circ h_T \circ \dots \tau_2 \circ h_2 \circ \tau_1\circ h_1 \text{ on } [0,1]^d\setminus F,$$
where $h_i(x) = g_i(x) - g_i(0)$ and $\tau_i(x) = x+g_i(0)$.

We may therefore factor $F$ as 
$$ F =  \sigma_R \circ \tau_T \circ h_T \circ \dots \circ \tau_1\circ h_1\circ  \sigma_R^{-1} \text{ on } [0,R]^d\setminus \sigma_R(F),$$
where $\sigma_R(x) = Rx$

Each $h_i$ is a $(1+\epsilon')$-bi-Lipschitz map fixing the origin, and therefore extends to a map $h_i\colon \mathbb{S}^d\rightarrow\mathbb{S}^d$ that is $(1+\epsilon)$-bi-Lipschitz in the spherical metric. Each $\tau_i$ is a translation of Euclidean space, and by Remark \ref{rmk:factorization} and \eqref{eq:f0}, the lengths of these translations are bounded depending only on $d$ and the distortion of $F$. Therefore, by Lemma \ref{lem:translationscaling} each $\tau_i$ factors into a controlled number of $(1+\epsilon)$-bi-Lipschitz homeomorphisms of the sphere. Finally, Lemma \ref{lem:translationscaling} says that the same holds for $\sigma_R$ and $\sigma_R^{-1}$.

Therefore, off of the set $E = A \cup (\mathbb{S}^d\setminus [0,R]^d)$, the original spherical homeomorphism $f$ agrees with the composition of a controlled number of $(1+\epsilon)$-bi-Lipschitz spherical homeomorphisms. The set $E\subseteq \mathbb{S}^d$ has $d$-dimensional Hausdorff measure less than $\delta$ in the sphere by our choices above.

\end{proof}

\printbibliography
\end{document}